\documentclass[11pt]{article}
\usepackage{amsmath, amssymb, amscd, epsfig, amsthm}
\usepackage{graphicx}
\usepackage{caption, subcaption}

\newtheorem{thm}{Theorem}
\newtheorem{prop}[thm]{Proposition}
\newtheorem{lem}[thm]{Lemma}

\newtheorem{cor}[thm]{Corollary}
\newtheorem{rem}[thm]{Remark}

\renewcommand{\epsilon}{\varepsilon}
\renewcommand{\phi}{\varphi}
\renewcommand{\deg}{\operatorname{deg}}

\newcommand{\BB}{\mathbb}
\newcommand{\g}{\mathfrak}
\newcommand{\separate}{\vskip5pt}

\newcommand{\re}{\operatorname{Re}}
\newcommand{\tr}{\operatorname{Tr}}

\newcommand{\B}{\overline}
\newcommand{\HC}{\BB H_{\BB C}}

\newcommand{\degt}{\widetilde{\operatorname{deg}}}

\usepackage[OT2,T1]{fontenc}
\newcommand\textcyr[1]{{\fontencoding{OT2}\fontfamily{wncyr}\selectfont #1}}
\newcommand{\Zh}{\textit{\textcyr{Zh}}}

\begin{document}

\title{\bf The Two-Loop Ladder Diagram and \\ Representations of $U(2,2)$}
\author{Matvei Libine}
\maketitle

\begin{abstract}
Feynman diagrams are a pictorial way of describing integrals predicting
possible outcomes of interactions of subatomic particles in the context of
quantum field physics. It is highly desirable to have an intrinsic
mathematical interpretation of Feynman diagrams, and in this article we
find the representation-theoretic meaning of a particular kind of Feynman
diagrams called the two-loop ladder diagram.
This is done in the context of representations of a Lie group $U(2,2)$,
its Lie algebra $\mathfrak{u}(2,2)$ and quaternionic analysis.
The results and techniques developed in this article are used in \cite{L}
to provide a mathematical interpretation of all conformal
four-point integrals -- including those described by the $n$-loop ladder
diagrams -- in the context of representations $U(2,2)$
and quaternionic analysis. Moreover, this representation-quaternionic model
produces a proof of ``magic identities'' in the Minkowski metric space.

No prior knowledge of physics or Feynman diagrams is assumed from the reader.
We provide a summary of all relevant results from quaternionic analysis to
make the article self-contained.
\end{abstract}



\section{Introduction}

Feynman diagrams are a pictorial way of describing integrals predicting
possible outcomes of interactions of subatomic particles in the context of
quantum field physics. As the number of variables which are being integrated
out increases, the integrals become more and more difficult to compute.
But in the cases when the integrals can be computed, the accuracy of
their prediction is amazing.
Feynman diagrams are also very interesting objects from mathematical
perspective, and a number of mathematicians are trying to find their intrinsic
mathematical meaning, mostly in the setting of algebraic geometry.
See, for example, \cite{Mar} for a summary of these algebraic-geometric
developments as well as a comprehensive list of references.
On the other hand, Igor Frenkel has noticed that at least some
types of Feynman diagrams can be interpreted in the context of representation
theory and quaternionic analysis.
Thus, in \cite{FL1, FL3} we give natural identifications of the two
fundamental Feynman diagrams shown in Figure \ref{basic} with
projectors onto irreducible components of certain representations of $U(2,2)$
in the context of quaternionic analysis.
For example, the one-loop Feynman diagram is identified with the projection
onto the first irreducible component $(\rho_1,\Zh^+)$ in the decomposition of
the tensor product of two representations of $\mathfrak{u}(2,2)$
into irreducible subrepresentations:
\begin{equation}  \label{decomp-intro}
(\pi^0_l, {\cal H}^+) \otimes (\pi^0_r, {\cal H}^+) \simeq
\bigoplus_{n=1}^{\infty} (\rho_n,\Zh^+\otimes \BB C^{n \times n}),
\end{equation}
where ${\cal H}^+$ denotes the space of harmonic functions on the algebra of
quaternions $\BB H$ (see the discussion after Remark \ref{discussion}).
Then we raise a natural question of finding mathematical interpretation of
other Feynman diagrams in the same setting.

\begin{figure}
\begin{center}
\begin{subfigure}{0.3\textwidth}
\centering
\includegraphics[scale=1]{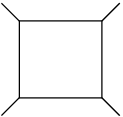}
\end{subfigure}
\begin{subfigure}{0.3\textwidth}
\centering
\includegraphics[scale=1]{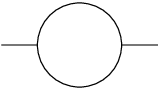}
\end{subfigure}
\end{center}
\caption{Feynman diagrams: the one-loop ladder diagram (left) and
the scalar vacuum polarization diagram (right).}
\label{basic}
\end{figure}

\begin{figure}
\begin{center}
\centerline{\includegraphics[scale=1]{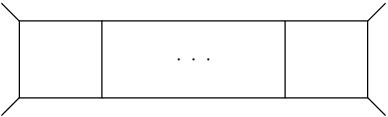}}
\end{center}
\caption{Conformal four-point box or ladder diagrams.}
\label{n-ladder}
\end{figure}

Conformal four-point box integrals play an important role in physics,
particularly Yang-Mills conformal field theory (see \cite{DHSS} and
references therein for more details).
These integrals are described by the box diagrams and have been thoroughly
studied by physicists.
For example, the integral described by the one-loop ladder diagram
is known to express the hyperbolic volume of an ideal tetrahedron
and is given by the dilogarithm function \cite{DD, W};
there are explicit expressions for the integrals described by the
ladder diagrams in terms of polylogarithms \cite{UD}.
Perhaps the most important property of the box integrals are the
``magic identities'' due to J.~M.~Drummond, J.~Henn, V.~A.~Smirnov
and E.~Sokatchev \cite{DHSS}.
These identities assert that all $n$-loop box integrals
for four scalar massless particles are equal to each other.
Thus we can parametrize the box integrals by the number of loops in the
diagrams and choose a single representative from the set of all $n$-loop
diagrams, such as the $n$-loop ladder diagram (Figure \ref{n-ladder}).

The original paper \cite{DHSS} gives a proof of magical identities for the
Euclidean metric case only and claims that the result is also true for the
Minkowski metric case.
In the Euclidean case, all variables belong to $\BB H$ and there are
no convergence issues whatsoever.
On the other hand, the Minkowski case (which is the case we consider)
is much more subtle.
In order to deal with convergence issues, we must
consider the so-called ``off-shell Minkowski integrals'' or
perturb the cycles of integration inside $\BB H \otimes_{\BB R} \BB C$.
Then the relative position of the cycles becomes very important.
In fact, choosing the ``wrong'' cycles typically results in integral being zero.

\begin{figure}
\begin{center}
\centerline{\includegraphics[scale=1]{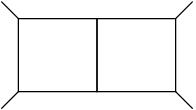}}
\end{center}
\caption{The two-loop ladder diagram.}
\label{2-ladder}
\end{figure}

In this paper we find the representation-theoretic meaning of the two-loop
ladder diagram (Figure \ref{2-ladder}) in the Minkowski metric case.
Thus, we associate to this diagram an integral operator $L^{(2)}$ on
${\cal H}^+ \otimes {\cal H}^+$, which is $\mathfrak{u}(2,2)$-equivariant.
We prove that the operator $L^{(2)}$ sends ${\cal H}^+ \otimes {\cal H}^+$
into itself and, in particular, that the result is a function of two variables
that is harmonic with respect to each variable, which is not at all obvious
from the construction. Then we show that if
$x \in {\cal H}^+ \otimes {\cal H}^+$ belongs to an irreducible component
isomorphic to $(\rho_n,\Zh^+ \otimes \BB C^{n \times n})$
in the decomposition (\ref{decomp-intro}), then
$$
L^{(2)}(x) = \mu_n x, \qquad \text{where} \qquad
\mu_n =
\begin{cases}
1 & \text{if $n=1$;} \\
\frac{(-1)^{n+1}}{n(n-1)} & \text{if $n \ge 2$.}
\end{cases}
$$
(Theorem \ref{main-thm}).
We also prove a certain non-obvious symmetry property for the operator $L^{(2)}$
(Lemma \ref{symmetry-lem}).
This property is a direct analogue of equation (8) in \cite{DHSS}
that is one of the ingredients of the proof of ``magic identities''.

In \cite{L} the results and techniques developed in this article are
extended to establish the ``magic identities'' for all conformal
four-point integrals in the Minkowski metric case. In particular, we
spell out the ``right'' choice of cycles of integration.
This is done by associating to each $n$-loop box integral an equivariant
operator $L^{(n)}$ on ${\cal H}^+ \otimes {\cal H}^+$ and computing the action
of $L^{(n)}$ on each irreducible component in the decomposition
(\ref{decomp-intro}).
It is reasonable to expect that an even larger class of Feynman diagrams
can be interpreted in the same context.

The paper is organized as follows. In Section \ref{preliminaries}
we establish our notations and state relevant results from quaternionic
analysis.
In Section \ref{Zh-decomp-section} we study the decomposition of a certain
representation  $(\varpi_2,\Zh)$ of $\mathfrak{u}(2,2)$ into irreducible
components (Theorem \ref{Zh2-decomposition} and
Proposition \ref{quotient-prop}).
These results are needed to establish that the operator $L^{(2)}$ is
$\mathfrak{u}(2,2)$-equivariant.
In Section \ref{fd-section} we describe the one- and two-loop ladder
integrals $l^{(1)}$ and $l^{(2)}$ represented by the one- and two-loop ladder
diagrams, then we introduce equivariant operators $L^{(1)}$ and $L^{(2)}$ on
${\cal H}^+ \otimes {\cal H}^+$ corresponding to those ladder integrals.
We also introduce auxiliary operators $\tilde L^{(2)}$ and $\mathring{L}^{(2)}$
closely related to $L^{(2)}$.
In Section \ref{I_R-section} we determine the action of the operator
$\tilde L^{(2)}$ by breaking it down as a composition of more elementary
operators and using their equivariance properties
(Proposition \ref{Lambda-prop} and Theorem \ref{main2-thm}).
Section \ref{main-section} contains our main result about the action of
the operator $L^{(2)}$ (Theorem \ref{main-thm}).
We essentially compute the action of $L^{(2)}$ on certain suitably chosen
generators of ${\cal H}^+ \otimes {\cal H}^+$ and reduce these calculations
to the ones already performed for $\tilde L^{(2)}$.
We also prove Lemma \ref{symmetry-lem} asserting a certain symmetry property
for the operator $L^{(2)}$.

The author was supported by the NSF grant DMS-0904612.

\section{Preliminaries}  \label{preliminaries}

In this section we establish notations and state relevant results from
quaternionic analysis. We mostly follow our previous papers \cite{FL1}
and \cite{FL2}.
A contemporary review of quaternionic analysis can be found in \cite{Su}.
Quaternionic analysis also has many applications in physics
(see, for instance, \cite{GT}).

\subsection{Complexified Quaternions $\HC$ and the Conformal Group $GL(2,\HC)$}

We recall some notations from \cite{FL1}.
Let $\HC$ denote the space of complexified quaternions:
$\HC = \BB H \otimes_{\BB R} \BB C$, it can be identified with the algebra of
$2 \times 2$ complex matrices:
$$
\HC = \BB H \otimes_{\BB R} \BB C \simeq \biggl\{
Z = \begin{pmatrix} z_{11} & z_{12} \\ z_{21} & z_{22} \end{pmatrix}
; \: z_{ij} \in \BB C \biggr\}
= \biggl\{ Z= \begin{pmatrix} z^0-iz^3 & -iz^1-z^2 \\ -iz^1+z^2 & z^0+iz^3
\end{pmatrix} ; \: z^k \in \BB C \biggr\}.
$$
For $Z \in \HC$, we write
$$
N(Z) = \det \begin{pmatrix} z_{11} & z_{12} \\ z_{21} & z_{22} \end{pmatrix}
= z_{11}z_{22}-z_{12}z_{21} = (z^0)^2 + (z^1)^2 + (z^2)^2 + (z^3)^2
$$
and think of it as the norm of $Z$. We realize $U(2)$ as
$$
U(2) = \{ Z \in \HC ;\: Z^*=Z^{-1} \},
$$
where $Z^*$ denotes the complex conjugate transpose of a complex matrix $Z$.
For $R>0$, we set
$$
U(2)_R = \{ RZ ;\: Z \in U(2) \} \quad \subset \HC
$$
and orient it as in \cite{FL1}, so that
$$
\int_{U(2)_R} \frac{dV}{N(Z)^2} = -2\pi^3 i,
$$
where $dV$ is a holomorphic 4-form
$$
dV = dz^0 \wedge dz^1 \wedge dz^2 \wedge dz^3
= \frac14 dz_{11} \wedge dz_{12} \wedge dz_{21} \wedge dz_{22}.
$$
Recall that a group $GL(2,\HC) \simeq GL(4,\BB C)$ acts on $\HC$ by fractional
linear (or conformal) transformations:
\begin{equation}  \label{conformal-action}
h: Z \mapsto (aZ+b)(cZ+d)^{-1} = (a'-Zc')^{-1}(-b'+Zd'),
\qquad Z \in \HC,
\end{equation}
where
$h = \bigl(\begin{smallmatrix} a & b \\ c & d \end{smallmatrix}\bigr)
\in GL(2,\HC)$ and 
$h^{-1} = \bigl(\begin{smallmatrix} a' & b' \\ c' & d' \end{smallmatrix}\bigr)$.

For convenience we recall Lemmas 10 and 61 from \cite{FL1}:

\begin{lem}  \label{Z-W}
For $h = \bigl( \begin{smallmatrix} a' & b' \\ c' & d' \end{smallmatrix} \bigr)
\in GL(2,\HC)$
with $h^{-1} =
\bigl(\begin{smallmatrix} a & b \\ c & d \end{smallmatrix}\bigr)$,
let $\tilde Z = (aZ+b)(cZ+d)^{-1}$ and $\tilde W = (aW+b)(cW+d)^{-1}$.
Then
\begin{align*}
(\tilde Z - \tilde W) &= (a'-Wc')^{-1} \cdot(Z-W) \cdot (cZ+d)^{-1}  \\
&= (a'-Zc')^{-1} \cdot(Z-W) \cdot (cW+d)^{-1}.
\end{align*}
\end{lem}

\begin{lem}  \label{Jacobian_lemma}
Let $d\tilde V$ denote the pull-back of $dV$ under the map
$Z \mapsto (aZ+b)(cZ+d)^{-1}$, where
$h = \bigl(\begin{smallmatrix} a' & b' \\ c' & d' \end{smallmatrix}\bigr)
\in GL(2,\HC)$ and
$h^{-1} = \bigl(\begin{smallmatrix} a & b \\ c & d \end{smallmatrix}\bigr)$.
Then
$$
dV = N(cZ+d)^2 \cdot N(a'-Zc')^2 \,d\tilde V.
$$
\end{lem}

\subsection{Harmonic Functions on $\HC$}

As in Section 2 of \cite{FL2}, we consider the space of $\BB C$-valued
functions on $\HC$ (possibly with singularities) which are holomorphic with
respect to the complex variables $z_{11},z_{12},z_{21},z_{22}$ and harmonic,
i.e. annihilated by
$$
\square 
= 4\biggl( \frac{\partial^2}{\partial z_{11}\partial z_{22}}
- \frac{\partial^2}{\partial z_{12}\partial z_{21}} \biggr)
= \frac{\partial^2}{(\partial z^0)^2} + \frac{\partial^2}{(\partial z^1)^2}
+ \frac{\partial^2}{(\partial z^2)^2} + \frac{\partial^2}{(\partial z^3)^2}.
$$
We denote this space by $\widetilde{\cal H}$.
Then the conformal group $GL(2,\HC)$ acts on $\widetilde{\cal H}$ by two
slightly different actions:
\begin{align*}
\pi^0_l(h): \: \phi(Z) \quad &\mapsto \quad \bigl( \pi^0_l(h)\phi \bigr)(Z) =
\frac 1{N(cZ+d)} \cdot \phi \bigl( (aZ+b)(cZ+d)^{-1} \bigr),  \\
\pi^0_r(h): \: \phi(Z) \quad &\mapsto \quad \bigl( \pi^0_r(h)\phi \bigr)(Z) =
\frac 1{N(a'-Zc')} \cdot \phi \bigl( (a'-Zc')^{-1}(-b'+Zd') \bigr),
\end{align*}
where
$h = \bigl(\begin{smallmatrix} a' & b' \\ c' & d' \end{smallmatrix}\bigr)
\in GL(2,\HC)$ and
$h^{-1} = \bigl(\begin{smallmatrix} a & b \\ c & d \end{smallmatrix}\bigr)$.
These two actions coincide on $SL(2,\HC) \simeq SL(4,\BB C)$
which is defined as the connected Lie subgroup of $GL(2,\HC)$ with Lie algebra
$$
\mathfrak{sl}(2,\HC) = \{ x \in \mathfrak{gl}(2,\HC) ;\: \re (\tr x) =0 \}
\simeq \mathfrak{sl}(4,\BB C).
$$

We introduce two spaces of harmonic polynomials:
$$
{\cal H}^+ = \widetilde{\cal H} \cap \BB C[z_{11},z_{12},z_{21},z_{22}],
$$
$$
{\cal H} = \widetilde{\cal H} \cap \BB C[z_{11},z_{12},z_{21},z_{22}, N(Z)^{-1}]
$$
and the space of harmonic polynomials regular at infinity:
$$
{\cal H}^- = \bigl\{ \phi \in \widetilde{\cal H};\:
N(Z)^{-1} \cdot \phi(Z^{-1}) \in {\cal H}^+ \bigr\}.
$$
Then
$$
{\cal H} = {\cal H}^- \oplus {\cal H}^+.
$$
In particular, there are no homogeneous harmonic functions in
$\BB C[z_{11},z_{12},z_{21},z_{22},N(Z)^{-1}]$ of degree $-1$.
Differentiating the actions $\pi^0_l$ and $\pi^0_r$, we obtain actions of
$\mathfrak{gl}(2,\HC) \simeq \mathfrak{gl}(4,\BB C)$ which preserve
the spaces ${\cal H}$, ${\cal H}^-$ and ${\cal H}^+$.
By abuse of notation, we denote these Lie algebra actions by
$\pi^0_l$ and $\pi^0_r$ respectively.
They are described in Subsection 3.2 of \cite{FL2}.

By Theorem 28 in \cite{FL1}, for each $R>0$, we have a
bilinear pairing between $(\pi^0_l, {\cal H})$ and $(\pi^0_r, {\cal H})$:
\begin{equation}  \label{H-pairing}
(\phi_1,\phi_2)_R = \frac 1{2\pi^2}
\int_{S^3_R} (\degt \phi_1)(Z) \cdot \phi_2(Z) \,\frac{dS}R,
\qquad \phi_1, \phi_2 \in {\cal H},
\end{equation}
where $S^3_R \subset \BB H$ is the three-dimensional sphere of radius $R$
centered at the origin
$$
S^3_R = \{ X \in \BB H ;\: N(X)=R^2 \},
$$
$dS$ denotes the usual Euclidean volume element on $S^3_R$, and
$\degt$ denotes the degree operator plus identity:
$$
\degt f = f + \deg f = f + z_{11}\frac{\partial f}{\partial z_{11}} +
z_{12}\frac{\partial f}{\partial z_{12}} + z_{21}\frac{\partial f}{\partial z_{21}}
+ z_{22}\frac{\partial f}{\partial z_{22}}.
$$
When this pairing is restricted to ${\cal H}^+ \times {\cal H}^-$,
it is $\mathfrak{gl}(2,\HC)$-invariant, independent of the choice of $R>0$,
non-degenerate and antisymmetric
$$
(\phi_1,\phi_2)_R = - (\phi_2,\phi_1)_R,
\qquad \phi_1 \in {\cal H}^+, \: \phi_2 \in {\cal H}^-.
$$

We conclude this subsection with an analogue of the Poisson formula
(Theorem 34 in \cite{FL1}). It involves a certain open region $\BB D^+_R$
in $\HC$ which will be defined in (\ref{D_R}).

\begin{thm} \label{Poisson}
Let $R>0$ and let $\phi \in \widetilde{\cal H}$ be a harmonic function
with no singularities on the closure of $\BB D^+_R$, then
$$
\phi(W) = \biggl( \phi,\frac1{N(Z-W)} \biggr)_R =
\frac 1{2\pi^2} \int_{Z \in S^3_R} \frac{(\degt \phi)(Z)}{N(Z-W)}
\,\frac{dS}R, \qquad \forall W \in \BB D^+_R.
$$
\end{thm}

\subsection{Representation $(\rho_1,\Zh)$ of $\mathfrak{gl}(2,\HC)$}

Let $\widetilde{\Zh}$ denote the space of $\BB C$-valued functions on $\HC$
(possibly with singularities) which are holomorphic with respect to the
complex variables $z_{11}$, $z_{12}$, $z_{21}$, $z_{22}$.
(There are no differential equations imposed on functions in $\widetilde{\Zh}$
whatsoever.) We recall the action of $GL(2,\HC)$ on $\widetilde{\Zh}$ given by
equation (49) in \cite{FL1}:
\begin{equation}  \label{1-action}
\rho_1(h): \: f(Z) \quad \mapsto \quad \bigl( \rho_1(h)f \bigr)(Z) =
\frac {f \bigl( (aZ+b)(cZ+d)^{-1} \bigr)}{N(cZ+d) \cdot N(a'-Zc')},
\end{equation}
where
$h = \bigl(\begin{smallmatrix} a' & b' \\ c' & d' \end{smallmatrix}\bigr)
\in GL(2,\HC)$ and 
$h^{-1} = \bigl(\begin{smallmatrix} a & b \\ c & d \end{smallmatrix}\bigr)$.
We have a natural $GL(2,\HC)$-equivariant multiplication map
\begin{equation}  \label{M}
M: (\pi_l^0, \widetilde{\cal H}) \otimes (\pi_r^0, \widetilde{\cal H})
\to (\rho_1,\widetilde{\Zh})
\end{equation}
which is determined on pure tensors by
$$
M \bigl( \phi_1(Z_1) \otimes \phi_2(Z_2) \bigr) = (\phi_1 \cdot \phi_2)(Z),
\qquad \phi_1, \phi_2 \in \widetilde{\cal H}.
$$
Differentiating the $\rho_1$-action, we obtain an action
(still denoted by $\rho_1$) of $\mathfrak{gl}(2,\HC)$ which preserves spaces
\begin{align}
\Zh^+ &= \{\text{polynomial functions on $\HC$}\}
= \BB C[z_{11},z_{12},z_{21},z_{22}] \qquad \text{and}  \label{zh+} \\
\Zh &= \bigl\{\text{polynomial functions on
$\{ Z \in \HC ;\: N(Z) \ne 0 \}$}\bigr\}
= \BB C[z_{11},z_{12},z_{21},z_{22}, N(Z)^{-1}].  \label{zh}
\end{align}

Recall Proposition 69 from \cite{FL1}:

\begin{prop}
The representation $(\rho_1,\Zh)$ of $\mathfrak{gl}(2,\HC)$
has a non-degenerate symmetric bilinear pairing
\begin{equation}  \label{pairing}
\langle f_1,f_2 \rangle =
\frac i{2\pi^3} \int_{Z \in U(2)_R} f_1(Z) \cdot f_2(Z) \,dV,
\qquad f_1, f_2 \in \Zh.
\end{equation}
This bilinear pairing is $\mathfrak{gl}(2,\HC)$-invariant and
independent of the choice of $R>0$.
\end{prop}

\subsection{The Group $\HC^{\times}$ and Its Matrix Coefficients}  \label{matrix-coeff-subsection}

We denote by $\HC^{\times}$ the group of invertible complexified quaternions:
$$
\HC^{\times} = \{ Z \in \HC ;\: N(Z) \ne 0 \}.
$$
Clearly, $\HC^{\times} \simeq GL(2,\BB C)$.
We denote by $(\tau_{\frac12},\BB S)$ the tautological representation of
$\HC^{\times}$. That is, we let
$$
\BB S = \biggl\{ \begin{pmatrix} s_1 \\ s_2 \end{pmatrix} ;
s_1, s_2 \in \BB C \biggr\}
$$
and define
$$
\tau_{\frac12}(Z) \begin{pmatrix} s_1 \\ s_2 \end{pmatrix}
= \begin{pmatrix} z_{11}s_1 + z_{12}s_2 \\ z_{21}s_1 + z_{22}s_2 \end{pmatrix},
\qquad
Z = \begin{pmatrix} z_{11} & z_{12} \\ z_{21} & z_{22} \end{pmatrix} \in \HC^{\times},
\quad \begin{pmatrix} s_1 \\ s_2 \end{pmatrix} \in \BB S.
$$
For $l=0,\frac12,1,\frac32, \dots$, we denote by $(\tau_l,V_l)$ the
$2l$-th symmetric power product of $(\tau_{\frac12},\BB S)$.
(In particular, $(\tau_0,V_0)$ is the trivial one-dimensional representation.)
Then each $(\tau_l,V_l)$ is an irreducible representation of $\HC^{\times}$
of dimension $2l+1$.
A concrete realization of $(\tau_l,V_l)$ as well as an isomorphism
$V_l \simeq \BB C^{2l+1}$ suitable for our purposes are described in
Subsection 2.5 of \cite{FL1}.

Recall the matrix coefficient functions of $\tau_l(Z)$ described by
equation (27) of \cite{FL1} (cf. \cite{V}):
\begin{equation}  \label{t}
t^l_{n\,\underline{m}}(Z) = \frac 1{2\pi i}
\oint (sz_{11}+z_{21})^{l-m} (sz_{12}+z_{22})^{l+m} s^{-l+n} \,\frac{ds}s,
\qquad
\begin{matrix} l = 0, \frac12, 1, \frac32, \dots, \\ m,n \in \BB Z +l, \\
 -l \le m,n \le l, \end{matrix}
\end{equation}
$Z=\bigl(\begin{smallmatrix} z_{11} & z_{12} \\
z_{21} & z_{22} \end{smallmatrix}\bigr) \in \HC$,
the integral is taken over a loop in $\BB C$ going once around the origin
in the counterclockwise direction.
These functions extend to $\HC$ as polynomials.
We have the following orthogonality relations with respect to the pairing
(\ref{H-pairing}):
\begin{equation}  \label{H-orthogonality}
\bigl( t^{l'}_{n'\,\underline{m'}}(Z), t^l_{m\underline{n}}(Z^{-1}) \cdot N(Z)^{-1} \bigr)_R
= -\bigl(t^l_{m\underline{n}}(Z^{-1}) \cdot N(Z)^{-1},t^{l'}_{n'\,\underline{m'}}(Z)\bigr)_R
= \delta_{ll'} \delta_{mm'} \delta_{nn'}
\end{equation}
and similar orthogonality relations with respect to the pairing (\ref{pairing}):
\begin{equation}  \label{orthogonality}
\bigl\langle t^{l'}_{n'\,\underline{m'}}(Z) \cdot N(Z)^{k'},
t^l_{m\underline{n}}(Z^{-1}) \cdot N(Z)^{-k-2} \bigr\rangle
= \frac1{2l+1} \delta_{kk'}\delta_{ll'} \delta_{mm'} \delta_{nn'},
\end{equation}
where the indices $k,l,m,n$ are
$l = 0, \frac12, 1, \frac32, \dots$, $m,n \in \BB Z +l$, $-l \le m,n \le l$,
$k \in \BB Z$ and similarly for $k',l',m',n'$ (see, for example, \cite{V}).
It is useful to recall that
$$
t^l_{m\underline{n}}(Z^{-1}) \quad \text{is proportional to} \quad
t^l_{-n\underline{-m}}(Z) \cdot N(Z)^{-2l}.
$$

One advantage of working with these functions is that they form $K$-type bases
of various spaces:

\begin{prop} [Proposition 19 in \cite{FL1}, Proposition 5 in \cite{FL3} and
Corollary 6 in \cite{FL3}] We have the following vector space bases:
\begin{enumerate}
\item
The functions 
$$
t^l_{n\,\underline{m}}(Z), \qquad
l=0, \frac12, 1, \frac32, \dots, \quad m,n=-l,-l+1,\dots,l,
$$
form a vector space basis of
${\cal H}^+ = \{ \phi \in \Zh^+;\: \square\phi=0 \}$;
\item
The functions 
$$
t^l_{n\,\underline{m}}(Z) \cdot N(Z)^{-(2l+1)}, \qquad
l=0, \frac12, 1, \frac32, \dots, \quad m,n=-l,-l+1,\dots,l,
$$
form a vector space basis of ${\cal H}^-$;
\item
The functions 
$$
t^l_{n\,\underline{m}}(Z) \cdot N(Z)^k, \qquad
l=0, \frac12, 1, \frac32, \dots, \quad m,n=-l,-l+1,\dots,l, \quad k=0,1,2,\dots,
$$
form a vector space basis of $\Zh^+ = \BB C[z_{11},z_{12},z_{21},z_{22}]$;
\item
The functions 
\begin{equation}  \label{Zh-basis}
t^l_{n\,\underline{m}}(Z) \cdot N(Z)^k, \qquad
l=0, \frac12, 1, \frac32, \dots, \quad m,n=-l,-l+1,\dots,l, \quad k \in\BB Z,
\end{equation}
form a vector space basis of $\Zh = \BB C[z_{11},z_{12},z_{21},z_{22}, N(Z)^{-1}]$.
\end{enumerate}
\end{prop}

Another advantage is having matrix coefficient expansions such as
those described in Propositions 25, 26 and 27 in \cite{FL1}.
For convenience we restate Proposition 25 from \cite{FL1}:

\begin{prop}
We have the following matrix coefficient expansion
\begin{equation}  \label{1/N-expansion}
\frac 1{N(Z-W)}= N(W)^{-1} \cdot \sum_{l,m,n}
t^l_{m\,\underline{n}}(Z) \cdot t^l_{n\,\underline{m}}(W^{-1}),
\qquad \begin{matrix} l=0,\frac 12, 1, \frac 32,\dots, \\
m,n = -l, -l+1, \dots, l, \end{matrix}
\end{equation}
which converges pointwise absolutely in the region
$\{ (Z,W) \in \HC \times \HC^{\times}; \: ZW^{-1} \in \BB D^+ \}$,
where $\BB D^+$ is an open region in $\HC$ to be defined in (\ref{D}).
\end{prop}

\subsection{Subgroups $U(2,2)_R \subset GL(2,\HC)$ and Domains
$\BB D^+_R$, $\BB D^-_R$}

We often regard the group $U(2,2)$ as a subgroup of $GL(2,\HC)$,
as described in Subsection 3.5 of \cite{FL1}. That is
$$
U(2,2) = \Biggl\{ \begin{pmatrix} a & b \\ c & d \end{pmatrix}
\in GL(2,\HC) ;\: a,b,c,d \in \HC,\:
\begin{matrix} a^*a = 1+c^*c \\ d^*d = 1+b^*b \\ a^*b=c^*d \end{matrix}
\Biggr\}.
$$
The maximal compact subgroup of $U(2,2)$ is
\begin{equation}  \label{U(2)xU(2)}
U(2) \times U(2) = \biggl\{
\begin{pmatrix} a & 0 \\ 0 & d \end{pmatrix} \in GL(2, \HC);\:
a,d \in \HC, \: a^*a=d^*d=1 \biggr\}.
\end{equation}

The group $U(2,2)$ acts on $\HC$ by fractional linear transformations
(\ref{conformal-action}) preserving $U(2) \subset \HC$ and open domains
\begin{equation}  \label{D}
\BB D^+ = \{ Z \in \HC;\: ZZ^*<1 \}, \qquad
\BB D^- = \{ Z \in \HC;\: ZZ^*>1 \},
\end{equation}
where the inequalities $ZZ^*<1$ and $ZZ^*>1$ mean that the matrix $ZZ^*-1$
is negative and positive definite respectively.
The sets $\BB D^+$ and $\BB D^-$ both have $U(2)$ as the Shilov boundary.

Similarly, for each $R>0$ we can define a conjugate of $U(2,2)$
$$
U(2,2)_R = \begin{pmatrix} R & 0 \\ 0 & 1 \end{pmatrix} U(2,2)
\begin{pmatrix} R^{-1} & 0 \\ 0 & 1 \end{pmatrix} \quad \subset GL(2,\HC).
$$
Each group $U(2,2)_R$ is a real form of $GL(2,\HC)$, preserves $U(2)_R$
and open domains
\begin{equation}  \label{D_R}
\BB D^+_R = \{ Z \in \HC ;\: ZZ^*<R^2 \}, \qquad
\BB D^-_R = \{ Z \in \HC ;\: ZZ^*>R^2 \}.
\end{equation}
These sets $\BB D^+_R$ and $\BB D^-_R$ both have $U(2)_R$ as the Shilov boundary.

\section{Representation $(\varpi_2,\Zh)$ and Its Properties}  \label{Zh-decomp-section}

In Sections \ref{I_R-section} and \ref{main-section} we break the two-loop
ladder diagram into smaller pieces and associate to each piece a
$\g{gl}(2,\HC)$-equivariant integral operator so that $L^{(2)}$ -- the operator
associated to the original two-loop ladder diagram -- is the composition of
the operators associated to the pieces. 
The intermediate operators that appear that way are equivariant with respect
to different actions of $\g{gl}(2,\HC)$, and all of these actions have appeared
before, for example, in \cite{FL3} with the exception of $\varpi_2$,
which we study in this section.

\subsection{Representations $(\varpi_m,\Zh)$}

In this subsection we introduce a family of representations $(\varpi_m,\Zh)$,
where the parameter $m=1,2,3,\dots$.
Thus we define the following actions of $GL(2,\HC)$ on $\widetilde{\Zh}$:
\begin{equation}  \label{pi_m-action}
\varpi_m(h): \: f(Z) \quad \mapsto \quad \bigl( \varpi_m(h)f \bigr)(Z) =
\frac {f \bigl( (aZ+b)(cZ+d)^{-1} \bigr)}{N(cZ+d)^m \cdot N(a'-Zc')},
\end{equation}
where
$h = \bigl(\begin{smallmatrix} a' & b' \\ c' & d' \end{smallmatrix}\bigr)
\in GL(2,\HC)$ and 
$h^{-1} = \bigl(\begin{smallmatrix} a & b \\ c & d \end{smallmatrix}\bigr)$.
When $m=1$, $\varpi_1$ coincides with $\rho_1$.
We have natural $GL(2,\HC)$-equivariant multiplication maps
$$
(\pi_l^0, \widetilde{\cal H}) \otimes (\varpi_m, \widetilde{\Zh})
\to (\varpi_{m+1},\widetilde{\Zh}) \quad \text{and} \quad
\underbrace{(\pi_l^0, \widetilde{\cal H}) \otimes \dots \otimes
(\pi_l^0, \widetilde{\cal H})}_{\text{$m$ times}} \otimes
(\pi_r^0, \widetilde{\cal H}) \to (\varpi_m,\widetilde{\Zh})
$$
which are determined on pure tensors by respectively
$$
\phi(Z_1) \otimes f(Z_2) \mapsto (\phi \cdot f)(Z) \quad \text{and} \quad
\phi_1(Z_1) \otimes \dots \otimes \phi_{m+1}(Z_{m+1}) \mapsto
(\phi_1 \cdot \ldots \cdot \phi_{m+1})(Z),
$$
where $\phi, \phi_1, \dots, \phi_{m+1} \in \widetilde{\cal H}$,
$f \in \widetilde{\Zh}$.
Differentiating the $\varpi_m$-action, we obtain an action of
$\mathfrak{gl}(2,\HC)$ (cf. Lemma 68 in \cite{FL1} which treats the case $m=1$).
Recall that $\partial = \bigl(\begin{smallmatrix} \partial_{11} & \partial_{21} \\
\partial_{12} & \partial_{22} \end{smallmatrix}\bigr)$, where
$\partial_{ij} = \frac{\partial}{\partial z_{ij}}$.

\begin{lem}  \label{pi_m-algebra-action}
The Lie algebra action $\varpi_m$ of $\mathfrak{gl}(2,\HC)$ on $\widetilde{\Zh}$
is given by
\begin{align*}
\varpi_m \begin{pmatrix} A & 0 \\ 0 & 0 \end{pmatrix} &:
f \mapsto \tr \bigl( A \cdot (-Z \cdot \partial f - f) \bigr)  \\
\varpi_m \begin{pmatrix} 0 & B \\ 0 & 0 \end{pmatrix} &:
f \mapsto \tr \bigl( B \cdot (-\partial f ) \bigr)  \\
\varpi_m \begin{pmatrix} 0 & 0 \\ C & 0 \end{pmatrix} &:
f \mapsto \tr \Bigl( C \cdot \bigl(
Z \cdot (\partial f) \cdot Z +(m+1)Zf \bigr) \Bigr)
= \tr \Bigl( C \cdot \bigl(Z \cdot \partial (Zf) \bigr) +(m-1)Zf \Bigr)  \\
\varpi_m \begin{pmatrix} 0 & 0 \\ 0 & D \end{pmatrix} &:
f \mapsto \tr \Bigl( D \cdot \bigl( (\partial f) \cdot Z + mf \bigr) \Bigr)
= \tr \Bigl( D \cdot \bigl( \partial (Zf) +(m-2)f \bigr) \Bigr).
\end{align*}
\end{lem}

This lemma implies that $\mathfrak{gl}(2, \HC)$ preserves spaces
$\Zh$ and $\Zh^+$ defined by (\ref{zh+})-(\ref{zh}).
In Subsection \ref{irred-comp-subsection} we extend this family of
representations $(\varpi_m,\Zh)$.
By Theorem \ref{JV-thm}, each $(\varpi_m,\Zh^+)$ is irreducible.

Define\footnote{Unfortunately, this notation $\Zh^-_m$ conflicts with notations
of Subsection 5.1 of \cite{FL1}.}
$$
\Zh^-_m = \biggl\{ f \in \Zh;\:
\varpi_m\begin{pmatrix} 0 & 1 \\ 1 & 0 \end{pmatrix} f(Z)
= \frac{f(Z^{-1})}{N(Z)^{m+1}} \in \Zh^+ \biggr\}.
$$
In the spirit of Definition 16 in \cite{FL1}, we can say that
$\Zh^-_m$ consists of those elements of $\Zh$ that are regular at infinity
according to the $\varpi_m$-action of $GL(2,\HC)$.
Clearly, $\Zh^-_m$ is invariant under the $\varpi_m$-action of
$\mathfrak{gl}(2,\HC)$ and $\Zh^+ \oplus \Zh^-_m$ are proper subspaces of $\Zh$.

\subsection{Irreducible Components of $(\varpi_2,\Zh)$}

In this subsection we are concerned with decomposition of
$(\varpi_2,\Zh)$ into irreducible components.

\begin{thm}  \label{Zh2-decomposition}
The spaces
\begin{align*}
\Zh^+ &= \BB C \text{-span of }
\bigl\{ t^l_{n\,\underline{m}}(Z) \cdot N(Z)^k;\: k \ge 0 \bigr\}, \\
\Zh_2^- &= \BB C \text{-span of }
\bigl\{ t^l_{n\,\underline{m}}(Z) \cdot N(Z)^k;\: k \le -(2l+3) \bigr\}, \\
I_2^- &= \BB C \text{-span of }
\bigl\{ t^l_{n\,\underline{m}}(Z) \cdot N(Z)^k;\: k \le -2 \bigr\}, \\
I_2^+ &= \BB C \text{-span of }
\bigl\{ t^l_{n\,\underline{m}}(Z) \cdot N(Z)^k;\: k \ge -(2l+1) \bigr\}, \\
J_2 &= \BB C \text{-span of }
\bigl\{ t^l_{n\,\underline{m}}(Z) \cdot N(Z)^k;\: -(2l+1) \le k \le -2 \bigr\}
\end{align*}
and their sums are the only proper $\mathfrak{gl}(2,\HC)$-invariant subspaces
of $\Zh$ (see Figure \ref{decomposition-fig}).

The irreducible components of $(\varpi_2, \Zh)$ are the subrepresentations
$$
(\varpi_2, \Zh^+), \qquad (\varpi_2, \Zh_2^-), \qquad (\varpi_2, J_2)
$$
and the quotients
\begin{equation}  \label{quotients}
\bigl( \varpi_2, \Zh/(I_2^- \oplus \Zh^+) \bigr)
= \bigl( \varpi_2, I_2^+/(\Zh^+ \oplus J_2) \bigr), \quad
\bigl( \varpi_2, \Zh/(\Zh_2^- \oplus I_2^+) \bigr)
= \bigl( \varpi_2, I_2^-/(\Zh_2^- \oplus J_2) \bigr).
\end{equation}
\end{thm}

\begin{figure}
\begin{center}
\setlength{\unitlength}{1mm}
\begin{picture}(120,70)
\multiput(10,10)(10,0){11}{\circle*{1}}
\multiput(10,20)(10,0){11}{\circle*{1}}
\multiput(10,30)(10,0){11}{\circle*{1}}
\multiput(10,40)(10,0){11}{\circle*{1}}
\multiput(10,50)(10,0){11}{\circle*{1}}
\multiput(10,60)(10,0){11}{\circle*{1}}

\thicklines
\put(70,0){\vector(0,1){70}}
\put(0,10){\vector(1,0){120}}

\thinlines
\put(68,10){\line(0,1){55}}
\put(70,8){\line(1,0){45}}
\qbezier(68,10)(68,8)(70,8)

\put(52,20){\line(0,1){45}}
\put(48.6,18.6){\line(-1,1){43.6}}
\qbezier(52,20)(52,15.2)(48.6,18.6)

\put(5,8){\line(1,0){35}}
\put(41.4,11.4){\line(-1,1){36.4}}
\qbezier(40,8)(44.8,8)(41.4,11.4)

\put(63,7){\line(1,0){52}}
\put(56,10){\line(-1,1){51}}
\qbezier(63,7)(59,7)(56,10)

\put(5,7){\line(1,0){45}}
\put(53,10){\line(0,1){55}}
\qbezier(50,7)(53,7)(53,10)

\put(72,67){$2l$}
\put(117,12){$k$}
\put(3,24){$\Zh_2^-$}
\put(33,62){$J_2$}
\put(112,44){$\Zh^+$}
\put(25,3){$I_2^-$}
\put(85,3){$I_2^+$}
\end{picture}
\end{center}
\caption{Decomposition of $(\varpi_2,\Zh)$ into irreducible components.}
\label{decomposition-fig}
\end{figure}
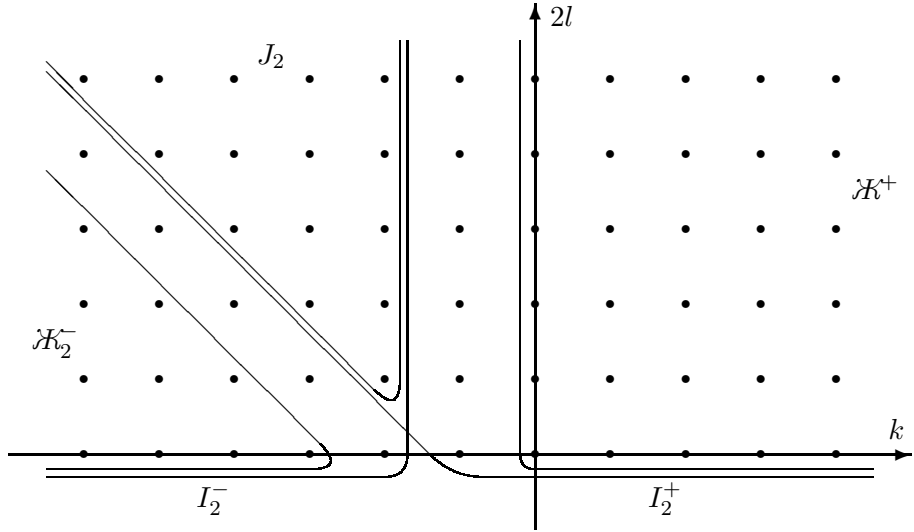

\begin{proof}
Note that the basis elements (\ref{Zh-basis}) consist of functions of the kind
$$
f_l(Z) \cdot N(Z)^k, \qquad \square f_l(Z)=0, \quad
l=0, \frac12, 1, \frac32, \dots, \quad k \in\BB Z,
$$
where the functions $f_l(Z)$ range over a basis of harmonic functions which
are polynomials of degree $2l$.
Recall that we consider $U(2) \times U(2)$ as a subgroup of $GL(2,\HC)$
via (\ref{U(2)xU(2)}).
For $k$ and $l$ fixed, these functions span an irreducible representation
of $U(2) \times U(2)$, which -- when restricted to $SU(2) \times SU(2)$ --
becomes isomorphic to $V_l \boxtimes V_l$, where $V_l$ denotes the
irreducible representation of $SU(2)$ of dimension $2l+1$ described in
Subsection \ref{matrix-coeff-subsection}.

To determine the effect of matrices of the kind
$\bigl(\begin{smallmatrix} 0 & B \\ 0 & 0 \end{smallmatrix}\bigr)
\in \mathfrak{gl}(2,\HC)$ with $B \in \HC$, we use
Lemma \ref{pi_m-algebra-action} describing their action and compute
$$
\partial \bigl( f_l(Z) \cdot N(Z)^k \bigr)
= \partial f_l \cdot N(Z)^k + k Z^+f_l \cdot N(Z)^{k-1}.
$$
By direct computation we have:
$$
\partial f_l \cdot N(Z) = Z^+ \deg f_l - Z^+ \cdot (\partial^+ f_l) \cdot Z^+
= 2l Z^+f_l - Z^+ \cdot (\partial^+ f_l) \cdot Z^+,
$$
$$
\square (Z^+f_l) = Z^+ \square f_l + 4\partial f_l
\qquad \text{and} \qquad
\square ( N(Z) \cdot g ) = N(Z) \cdot \square g + 4(\deg + 2) g.
$$
Hence we can write
\begin{equation}  \label{Z^+f}
Z^+f_l = \Bigl(Z^+f_l - \frac{\partial f_l \cdot N(Z)}{2l+1} \Bigr)
+ \frac{\partial f_l \cdot N(Z)}{2l+1}
= \frac{Z^+ \cdot (\partial^+ f_l) \cdot Z^+ + Z^+f_l}{2l+1}
+ \frac{\partial f_l \cdot N(Z)}{2l+1}
\end{equation}
and
\begin{equation}  \label{B-action}
\partial \bigl( f_l(Z) \cdot N(Z)^k \bigr)
= \frac{2l+k+1}{2l+1} \partial f_l \cdot N(Z)^k
+ \frac{k}{2l+1} \bigl( Z^+ \cdot (\partial^+ f_l) \cdot Z^+ + Z^+f_l \bigr)
\cdot N(Z)^{k-1}
\end{equation}
with $\partial f_l$ and $Z^+ \cdot (\partial^+ f_l) \cdot Z^+ + Z^+f_l$
being harmonic and having degrees $2l-1$ and $2l+1$ respectively.

Next we determine the effect of matrices of the kind
$\bigl(\begin{smallmatrix} 0 & 0 \\ C & 0 \end{smallmatrix}\bigr)
\in \mathfrak{gl}(2,\HC)$ with $C \in \HC$. Again, we use
Lemma \ref{pi_m-algebra-action} and compute
$$
Z \cdot \partial \bigl( f_l \cdot N(Z)^k \bigr) \cdot Z + 3 Zf_l \cdot N(Z)^k
= Z \cdot (\partial f_l) \cdot Z \cdot N(Z)^k
+ (k+3) Zf_l \cdot N(Z)^k.
$$
Conjugating (\ref{Z^+f}) we see that
$$
Zf_l = \frac{Z \cdot (\partial f_l) \cdot Z + Zf_l}{2l+1}
+ \frac{\partial^+ f_l \cdot N(Z)}{2l+1}.
$$
Therefore,
\begin{multline}  \label{C-action}
Z \cdot \partial \bigl( f_l \cdot N(Z)^k \bigr) \cdot Z + 3 Zf_l \cdot N(Z)^k \\
= \frac{2l+k+3}{2l+1} \bigl( Z \cdot (\partial f_l) \cdot Z + Zf_l \bigr)
\cdot N(Z)^k + \frac{k+2}{2l+1} \partial^+ f_l \cdot N(Z)^{k+1}
\end{multline}
with $Z \cdot (\partial f_l) \cdot Z + Zf_l$ and $\partial^+ f_l$
being harmonic and having degrees $2l+1$ and $2l-1$ respectively.

The actions of
$\bigl(\begin{smallmatrix} A & 0 \\ 0 & 0 \end{smallmatrix}\bigr)$,
$\bigl(\begin{smallmatrix} 0 & B \\ 0 & 0 \end{smallmatrix}\bigr)$,
$\bigl(\begin{smallmatrix} 0 & 0 \\ C & 0 \end{smallmatrix}\bigr)$ and
$\bigl(\begin{smallmatrix} 0 & 0 \\ 0 & D \end{smallmatrix}\bigr)$
are illustrated in Figure \ref{actions}.
In the diagram describing
$\varpi_2\bigl(\begin{smallmatrix} 0 & B \\ 0 & 0 \end{smallmatrix}\bigr)$
the vertical arrow disappears if $l=0$ or $2l+k+1=0$
and the diagonal arrow disappears if $k=0$.
Similarly, in the diagram describing
$\varpi_2\bigl(\begin{smallmatrix} 0 & 0 \\ C & 0 \end{smallmatrix}\bigr)$
the vertical arrow disappears if $2l+k+3=0$ and the
diagonal arrow disappears if $k=-2$ or $l=0$.
This proves that $\Zh^+$, $\Zh_2^-$, $I_2^+$, $I_2^-$ and $J_2$
are $\mathfrak{gl}(2,\HC)$-invariant subspaces of $\Zh$. Note that 
$$
\tr (Z \cdot \partial f + f) =
\tr \begin{pmatrix}
z_{11} \partial_{11}f + z_{12} \partial_{12}f + f & \ast \ast \ast \\
\ast \ast \ast & z_{21} \partial_{21}f + z_{22} \partial_{22}f + f
\end{pmatrix}
= (\deg+2)f,
$$
hence
$Z \cdot (\partial f_l) \cdot Z + Zf_l = (Z \cdot \partial f_l + f_l) \cdot Z$
and its conjugate $Z^+ \cdot (\partial^+ f_l) \cdot Z^+ + Z^+f_l$ are never zero.
It follows from (\ref{B-action}) and (\ref{C-action}) that the
subrepresentations $(\varpi_2, \Zh^+)$, $(\varpi_2, \Zh_2^-)$, $(\varpi_2, J_2)$
and the quotients (\ref{quotients}) are irreducible with respect to the
$\varpi_2$-action of $\mathfrak{gl}(2,\HC)$.
Moreover, $\Zh^+$, $\Zh_2^-$, $I_2^+$, $I_2^-$, $J_2$ and their sums are
the only proper $\mathfrak{gl}(2,\HC)$-invariant subspaces of $\Zh$.
\end{proof}

\begin{figure}
\begin{center}
\begin{subfigure}[b]{0.22\textwidth}
\centering
\includegraphics[scale=0.2]{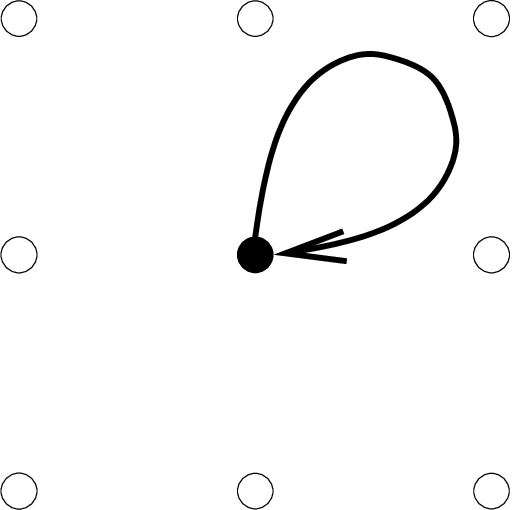}
\caption{Action of
$\varpi_2\bigl(\begin{smallmatrix} A & 0 \\ 0 & 0 \end{smallmatrix}\bigr)$.}
\end{subfigure}
\quad
\begin{subfigure}[b]{0.22\textwidth}
\centering
\includegraphics[scale=0.2]{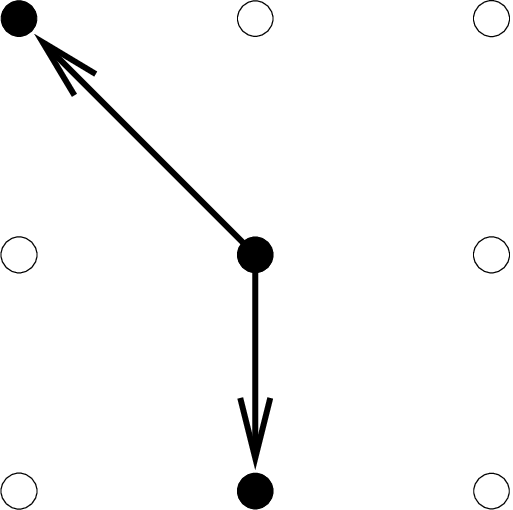}
\caption{Action of
$\varpi_2\bigl(\begin{smallmatrix} 0 & B \\ 0 & 0 \end{smallmatrix}\bigr)$.}
\end{subfigure}
\quad
\begin{subfigure}[b]{0.22\textwidth}
\centering
\includegraphics[scale=0.2]{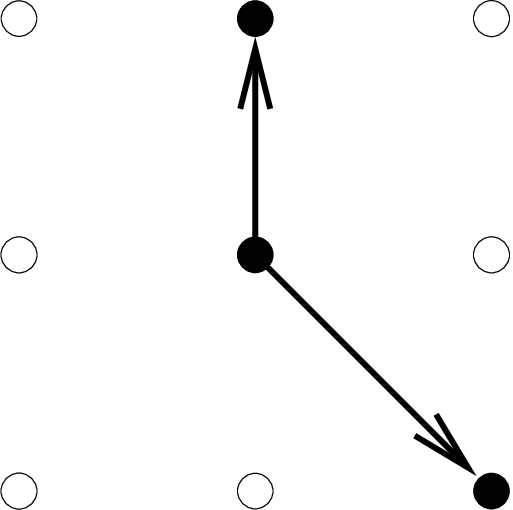}
\caption{Action of
$\varpi_2\bigl(\begin{smallmatrix} 0 & 0 \\ C & 0 \end{smallmatrix}\bigr)$.}
\end{subfigure}
\quad
\begin{subfigure}[b]{0.22\textwidth}
\centering
\includegraphics[scale=0.2]{AD.eps}
\caption{Action of
$\varpi_2\bigl(\begin{smallmatrix} 0 & 0 \\ 0 & D \end{smallmatrix}\bigr)$.}
\end{subfigure}
\end{center}
\caption{}
\label{actions}
\end{figure}


\begin{rem}
The same argument can be used to identify the subrepresentations and
irreducible components of all $(\varpi_m, \Zh)$'s.
\end{rem}

Next we identify the quotient representations (\ref{quotients}).

\begin{prop}  \label{quotient-prop}
As representations of $\mathfrak{gl}(2,\HC)$,
$$
\bigl( \varpi_2, \Zh/(I_2^- \oplus \Zh^+) \bigr) \simeq (\pi^0_l, {\cal H}^+)
\qquad \text{and} \qquad
\bigl( \varpi_2, \Zh/(\Zh_2^- \oplus I_2^+) \bigr) \simeq (\pi^0_l, {\cal H}^-),
$$
in both cases the isomorphism map being
$$
{\cal H}^{\pm} \ni \phi(Z) \quad \mapsto \quad
\frac{\degt \phi(Z)}{N(Z)} \in
\begin{matrix} \Zh/(\Zh_2^- \oplus I_2^+) \\ \text{or} \\
\Zh/(I_2^- \oplus \Zh^+). \end{matrix}
$$
The inverse of this isomorphism is given by
\begin{equation}  \label{inverse-iso}
\begin{matrix} \Zh/(\Zh_2^- \oplus I_2^+) \\ \text{or} \\
\Zh/(I_2^- \oplus \Zh^+) \end{matrix} \ni f(Z)
\quad \mapsto \quad \biggl\langle f(Z), \frac1{N(Z-W)} \biggr\rangle_Z
= \frac i{2\pi^3} \int_{Z \in U(2)_R} \frac{f(Z)\,dV}{N(Z-W)} \in {\cal H}.
\end{equation}
\end{prop}

\begin{proof}
First we check that this vector space isomorphism commutes with the action of
diagonal matrices.
Let $h=\bigl(\begin{smallmatrix} a & 0 \\ 0 & d \end{smallmatrix}\bigr)^{-1}
\in GL(2,\HC)$, then
$$
\varpi_2(h): \frac{\phi(Z)}{N(Z)} \quad \mapsto \quad
\frac{N(a)}{N(d)^2} \cdot \frac{\phi(aZd^{-1})}{N(aZd^{-1})}
= \frac1{N(d)} \cdot \frac{\phi(aZd^{-1})}{N(Z)}
= \frac1{N(Z)} \cdot (\pi^0_l \phi)(Z).
$$

Next we check for the matrices of the kind
$\bigl(\begin{smallmatrix} 0 & B \\ 0 & 0 \end{smallmatrix}\bigr)
\in \mathfrak{gl}(2,\HC)$ with
$B \in \HC$. Their action is described in Lemma \ref{pi_m-algebra-action}.
Suppose that $\phi \in {\cal H}$ is homogeneous of homogeneity degree $\lambda$
(note that $\lambda$ is never equal to $-1$).
From (\ref{B-action}) with $k=-1$ we see that
$$
\partial \Bigl( \frac{\phi}{N(Z)} \Bigr)
\equiv \frac{\lambda}{\lambda+1}  \frac{\partial \phi}{N(Z)}
\mod \Zh^+ \oplus J_2 \oplus \Zh_2^-,
$$
which proves that the isomorphism respects the actions of the matrices
$\bigl(\begin{smallmatrix} 0 & B \\ 0 & 0 \end{smallmatrix}\bigr)
\in \mathfrak{gl}(2,\HC)$.

Then we check for the matrices of the kind
$\bigl(\begin{smallmatrix} 0 & 0 \\ C & 0 \end{smallmatrix}\bigr)
\in \mathfrak{gl}(2,\HC)$ with
$C \in \HC$. Suppose again that $\phi \in {\cal H}$ is homogeneous of
homogeneity degree $\lambda$.
From Lemma \ref{pi_m-algebra-action} and (\ref{C-action})
with $k=-1$ we see that
$$
Z \cdot \partial \Bigl( \frac{\phi}{N(Z)} \Bigr) \cdot Z
+ 3\frac{Z\phi}{N(Z)} \equiv
\frac{\lambda+2}{\lambda+1} \frac{Z \cdot (\partial \phi) \cdot Z +Z\phi}{N(Z)}
\mod \Zh^+ \oplus J_2 \oplus \Zh_2^-,
$$
which proves that the isomorphism respects the actions of the matrices
$\bigl(\begin{smallmatrix} 0 & 0 \\ C & 0 \end{smallmatrix}\bigr)
\in \mathfrak{gl}(2,\HC)$.

Finally, it follows from the matrix coefficient expansion (\ref{1/N-expansion})
and the orthogonality relations (\ref{orthogonality}) that the map
(\ref{inverse-iso}) is well defined and is the inverse isomorphism.
\end{proof}

\subsection{Invariant Pairing between $(\varpi_2,\Zh)$ and $(\pi_r^0, \Zh)$}

We can extend the $\pi_r^0$ action of $GL(2,\HC)$ on $\widetilde{\cal H}$
to $\widetilde{\Zh}$.
Differentiating this action, we obtain an action of $\mathfrak{gl}(2,\HC)$,
which preserves $\Zh$, $\Zh^+$ (and, of course, ${\cal H}^-$, ${\cal H}^+$).
This action is given by the same formulas as in Subsection 3.2 of \cite{FL2}.
Then we have a bilinear pairing between $(\varpi_2,\Zh)$ and $(\pi_r^0, \Zh)$
that is formally the same as (\ref{pairing}):
\begin{equation}  \label{pairing2}
\langle f_1,f_2 \rangle =
\frac i{2\pi^3} \int_{Z \in U(2)_R} f_1(Z) \cdot f_2(Z) \,dV,
\qquad R>0,
\end{equation}
except now the $\mathfrak{gl}(2,\HC)$-actions on the first and second
components are different: $f_1 \in (\varpi_2,\Zh)$ and $f_2 \in (\pi_r^0, \Zh)$.
This bilinear pairing is $\mathfrak{gl}(2,\HC)$-invariant, non-degenerate and
independent of the choice of $R>0$.
In other words, the representations $(\varpi_2,\Zh)$ and $(\pi_r^0, \Zh)$
are dual to each other.
The proof of these assertions is exactly the same as that of
Proposition 69 in \cite{FL1}.

Now, let us restrict $f_2$ to $(\pi_r^0,{\cal H}) \subset (\pi_r^0, \Zh)$.
Then, by (\ref{orthogonality}), this pairing annihilates all
$f_1 \in (\varpi_2, \Zh_2^- \oplus J_2 \oplus \Zh^+)$.
Hence this pairing descends to a pairing between $(\pi_r^0,{\cal H})$ and
$\bigl(\varpi_2, \Zh/(\Zh_2^- \oplus J_2 \oplus \Zh^+)\bigr)$.
By Proposition \ref{quotient-prop}, the latter representation is isomorphic
to $(\pi_l^0,{\cal H})$.
Thus we obtain the following expression for a $\mathfrak{gl}(2,\HC)$-invariant
bilinear pairing between $(\pi_l^0,{\cal H})$ and $(\pi_r^0,{\cal H})$:
\begin{equation}  \label{H-pairing2}
(\phi_1,\phi_2) =
\frac i{2\pi^3} \int_{Z \in U(2)_R} (\degt\phi_1)(Z) \cdot \phi_2(Z)
\,\frac{dV}{N(Z)}, \qquad \phi_1,\phi_2 \in {\cal H}.
\end{equation}
(This pairing is independent of the choice of $R>0$.)
Comparing the orthogonality relations (\ref{H-orthogonality}) and
(\ref{orthogonality}), we see that the pairings (\ref{H-pairing})
and (\ref{H-pairing2}) coincide when $\phi_1 \in {\cal H}^+$,
$\phi_2 \in {\cal H}^-$ (but differ for other choices of
$\phi_1$ and $\phi_2$).

\subsection{Multiplication Maps and Their Images}

In \cite{FL3} we prove the following result, its proof is very similar to
that of Theorem \ref{Zh2-decomposition}:

\begin{thm} [Theorem 7 in \cite{FL3}]
The representation $(\rho_1,\Zh)$ of $\mathfrak{gl}(2,\HC)$ has the
following decomposition into irreducible components:
$$
(\rho_1,\Zh) = (\rho_1,\Zh_1^-) \oplus (\rho_1,\Zh^0) \oplus (\rho_1,\Zh^+),
$$
where
\begin{align*}
\Zh^+ &= \BB C \text{-span of }
\bigl\{ t^l_{n\,\underline{m}}(Z) \cdot N(Z)^k;\: k \ge 0 \bigr\}, \\
\Zh_1^- &= \BB C \text{-span of }
\bigl\{ t^l_{n\,\underline{m}}(Z) \cdot N(Z)^k;\: k \le -(2l+2) \bigr\}, \\
\Zh^0 &= \BB C \text{-span of }
\bigl\{ t^l_{n\,\underline{m}}(Z) \cdot N(Z)^k;\: -(2l+1) \le k \le -1 \bigr\}
\end{align*}
(see Figure \ref{decomposition-fig2}).
\end{thm}

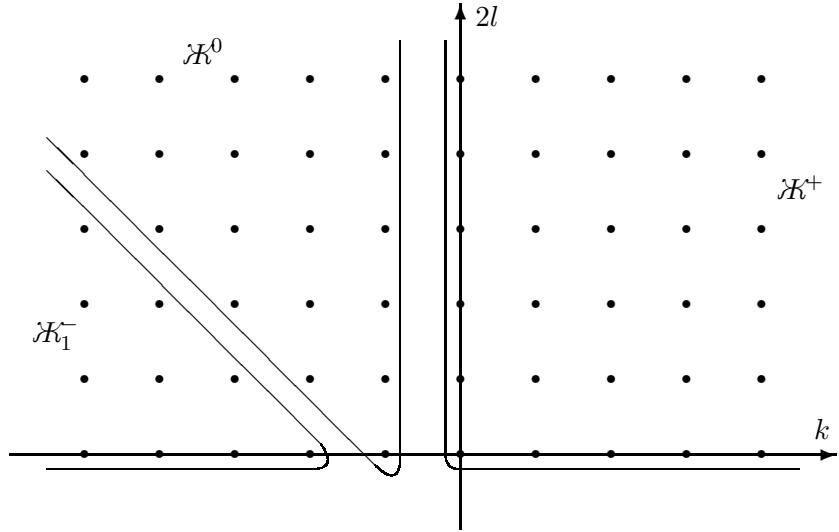
\begin{figure}
\begin{center}
\setlength{\unitlength}{1mm}
\begin{picture}(110,70)
\multiput(10,10)(10,0){10}{\circle*{1}}
\multiput(10,20)(10,0){10}{\circle*{1}}
\multiput(10,30)(10,0){10}{\circle*{1}}
\multiput(10,40)(10,0){10}{\circle*{1}}
\multiput(10,50)(10,0){10}{\circle*{1}}
\multiput(10,60)(10,0){10}{\circle*{1}}
\thicklines
\put(60,0){\vector(0,1){70}}
\put(0,10){\vector(1,0){110}}
\thinlines
\put(58,10){\line(0,1){55}}
\put(60,8){\line(1,0){45}}
\put(52,10){\line(0,1){55}}
\put(5,8){\line(1,0){35}}
\put(41.4,11.4){\line(-1,1){36.4}}
\put(48.6,8.6){\line(-1,1){43.6}}
\qbezier(58,10)(58,8)(60,8)
\qbezier(40,8)(43.8,8)(41.4,11.4)
\qbezier(52,10)(52,5.2)(48.6,8.6)
\put(62,67){$2l$}
\put(107,12){$k$}
\put(3,25){$\Zh_1^-$}
\put(23,62){$\Zh^0$}
\put(102,44){$\Zh^+$}
\end{picture}
\end{center}
\caption{Decomposition of $(\rho_1,\Zh)$ into irreducible components.}
\label{decomposition-fig2}
\end{figure}

Recall the natural $\mathfrak{gl}(2,\HC)$-equivariant multiplication maps:
$$
(\pi^0_l, {\cal H}^\pm) \otimes (\pi^0_r, {\cal H}^\pm) \to (\rho_1, \Zh)
$$
sending pure tensors
$$
\phi_1(Z_1) \otimes \phi_2(Z_2) \mapsto (\phi_1 \cdot \phi_2)(Z).
$$

\begin{lem}[Lemma 8 in \cite{FL3}]  \label{image-lemma}
Under the multiplication maps
$(\pi^0_l, {\cal H}^\pm) \otimes (\pi^0_r, {\cal H}^\pm) \to (\rho_1, \Zh)$,
\begin{enumerate}
\item
The image of ${\cal H}^+ \otimes {\cal H}^+$ in $\Zh$ is $\Zh^+$;
\item
The image of ${\cal H}^- \otimes {\cal H}^-$ in $\Zh$ is $\Zh_1^-$;
\item
The image of ${\cal H}^- \otimes {\cal H}^+$ in $\Zh$ is $\Zh^0$.
\end{enumerate}
\end{lem}

We turn our attention to the images under the natural
$\mathfrak{gl}(2,\HC)$-equivariant multiplication maps
$$
(\pi^0_l, {\cal H}^\pm) \otimes (\pi^0_l, {\cal H}^\pm) \otimes
(\pi^0_r, {\cal H}^\pm) \to (\varpi_2, \Zh)
\quad \text{and} \quad
(\pi^0_l, {\cal H}^\pm) \otimes (\varpi_1, V_i) \to (\varpi_2, \Zh),
$$
where $V_i$ ranges over the irreducible subrepresentations of
$(\rho_1, \Zh)$, i.e. $\Zh^+$, $\Zh_1^-$ and $\Zh^0$.

\begin{prop}
Under the multiplication maps
$(\pi^0_l, {\cal H}^\pm) \otimes (\pi^0_l, {\cal H}^\pm)
\otimes (\pi^0_r, {\cal H}^\pm) \to (\varpi_2, \Zh)$ and
$(\pi^0_l, {\cal H}^\pm) \otimes (\rho_1, V_i) \to (\varpi_2, \Zh)$,
where $V_i = \Zh^+, \Zh_1^-, \Zh^0$,
\begin{enumerate}
\item
The images of ${\cal H}^+ \otimes {\cal H}^+ \otimes {\cal H}^+$ and
${\cal H}^+ \otimes \Zh^+$ in $\Zh$ are $\Zh^+$;
\item
The images of ${\cal H}^- \otimes {\cal H}^- \otimes {\cal H}^-$ and
${\cal H}^- \otimes \Zh_1^-$ in $\Zh$ are $\Zh_2^-$;
\item
The images of ${\cal H}^+ \otimes {\cal H}^+ \otimes {\cal H}^-$,
${\cal H}^- \otimes \Zh^+$ and ${\cal H}^+ \otimes \Zh^0$ in $\Zh$ are $I_2^+$;
\item
The images of ${\cal H}^- \otimes {\cal H}^- \otimes {\cal H}^+$,
${\cal H}^+ \otimes \Zh^-_1$ and ${\cal H}^- \otimes \Zh^0$ in $\Zh$ are $I_2^-$.
\end{enumerate}
\end{prop}

\begin{proof}
By Lemma \ref{image-lemma}, the multiplication map
${\cal H}^+ \otimes {\cal H}^+ \otimes {\cal H}^+ \to \Zh$
factors through the multiplication map
${\cal H}^+ \otimes \Zh^+ \to \Zh$, hence they have the same images.
Since the product of polynomials is another polynomial, this image lies in
$\Zh^+$. The representation $(\varpi_2, \Zh)$ is irreducible,
so the image is all of $\Zh^+$.

By Lemma \ref{image-lemma}, the map
${\cal H}^+ \otimes {\cal H}^+ \otimes {\cal H}^- \to \Zh$
factors through the maps
${\cal H}^- \otimes \Zh^+ \to \Zh$ and ${\cal H}^+ \otimes \Zh^0 \to \Zh$,
hence they have the same images. Let us denote this image by $\widetilde{I}$.
Clearly, $\widetilde{I}$ contains the function $N(Z)^{-1}$,
which generates $I_2^+$, thus $I_2^+ \subset \widetilde{I}$.
It remains to show that $\widetilde{I} \subset I_2^+$.
By Theorem \ref{Zh2-decomposition}, if $I_2^+ \subsetneq \widetilde{I}$,
then $\widetilde{I}$ also contains $\Zh_2^-$ and hence functions
$N(Z)^k$ with $k \le -3$.
Thus it is sufficient to prove that $\widetilde{I}$ cannot contain
$N(Z)^{-3}$.

By construction, $\widetilde{I}$ is spanned by
\begin{equation}  \label{N(Z)^{-3}}
t^l_{n\,\underline{m}}(Z) \cdot t^{l'}_{n'\,\underline{m'}}(Z) \cdot N(Z)^{k-2l'-1} ,
\qquad k \ge 0.
\end{equation}
Note that if $V_l$ and $V_{l'}$ are two irreducible representations of $SU(2)$
of dimensions $2l+1$ and $2l'+1$ respectively, then their tensor product
contains a copy of the trivial representation if and only if $l=l'$.
This means that a linear combination of the functions (\ref{N(Z)^{-3}})
can express $N(Z)^{-3}$ only if $l=l'$.
But then the homogeneity degree of (\ref{N(Z)^{-3}}) is $2(k-1) \ge -2$.
Therefore, $N(Z)^{-3} \notin \widetilde{I}$.

Finally the remaining parts of the proposition follow by applying
$\bigl(\begin{smallmatrix} 0 & 1 \\ 1 & 0 \end{smallmatrix}\bigr)$
to the assertions we have proved. For example, applying
$(\pi^0_l \otimes \rho_1)
\bigl(\begin{smallmatrix} 0 & 1 \\ 1 & 0 \end{smallmatrix}\bigr)$
to the left hand side of ${\cal H}^- \otimes \Zh^+ \to I_2^+$ and
$\varpi_2 \bigl(\begin{smallmatrix} 0 & 1 \\ 1 & 0 \end{smallmatrix}\bigr)$
to the right hand side, we see that the image of
${\cal H}^+ \otimes \Zh^-_1$ is $I_2^-$.
\end{proof}

\section{The One-Loop and Two-Loop Ladder Diagrams}  \label{fd-section}

In this section we introduce the integrals $l^{(1)}$ and $l^{(2)}$ represented
by the one- and two-loop ladder diagrams.
Then we introduce the integral operators $L^{(1)}$ and $L^{(2)}$ on
${\cal H}^+ \otimes {\cal H}^+$.
We also introduce auxiliary integral operators $\tilde L^{(2)}$ and
$\mathring{L}^{(2)}$ closely related to $L^{(2)}$.

\subsection{Ladder Integrals}

\begin{figure}
\begin{center}
\begin{subfigure}{0.3\textwidth}
\centering
\setlength{\unitlength}{1mm}
\begin{picture}(30,28)
\put(15,14){\circle*{2}}
\put(15,4){\line(0,1){20}}
\put(5,14){\line(1,0){20}}
\put(0,13){$Z_2$}
\put(14,25){$Z_1$}
\put(26,13){$W_2$}
\put(14,0){$W_1$}
\put(17,15){$T$}
\end{picture}
\end{subfigure}
\qquad
\begin{subfigure}{0.3\textwidth}
\centering
\setlength{\unitlength}{1mm}
\begin{picture}(40,28)
\put(14,14){\circle*{2}}
\put(26,14){\circle*{2}}
\put(6,14){\line(1,0){28}}
\put(14,14){\line(3,5){6}}
\put(14,14){\line(3,-5){6}}
\put(26,14){\line(-3,5){6}}
\put(26,14){\line(-3,-5){6}}
\multiput(20,4)(0,2){10}{\line(0,1){1}}
\put(0,13){$Z_2$}
\put(19,25){$Z_1$}
\put(36,13){$W_2$}
\put(19,0){$W_1$}
\put(27,16){$T_2$}
\put(10,16){$T_1$}
\end{picture}
\end{subfigure}
\end{center}
\caption{The one-loop ladder diagram (left) and
the two-loop ladder diagram (right).}
\label{12ladder}
\end{figure}
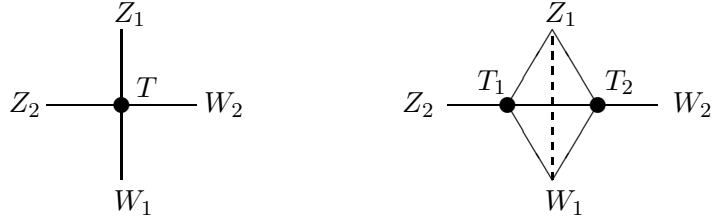



As in \cite{DHSS}, we use the coordinate space variable notation
(as opposed to the momentum notation).
With this choice of variable notation, the one- and two-loop ladder diagrams
are represented as in Figure \ref{12ladder}.
The one-loop ladder integral is
$$
l^{(1)}(Z_1,Z_2;W_1,W_2) =
\frac i{2\pi^3} \int_{T \in U(2)}
\frac{dV}{N(Z_1-T) \cdot N(Z_2-T) \cdot N(W_1-T) \cdot N(W_2-T)}.
$$
Next, we have the two-loop ladder integral:
\begin{equation}  \label{l2-rel}
l^{(2)}(Z_1,Z_2;W_1,W_2) = N(Z_1-W_1) \cdot \tilde l^{(2)}(Z_1,Z_2;W_1,W_2),
\end{equation}
where
\begin{multline*}
-4\pi^6 \cdot \tilde l^{(2)}(Z_1,Z_2;W_1,W_2) \\
= \iint_{\genfrac{}{}{0pt}{}{T_1 \in U(2)_{r_1}}{T_2 \in U(2)_{r_2}}}
\frac{|T_1-T_2|^{-2} \, dV_{T_1} \, dV_{T_2}}{|Z_1-T_1|^2 \cdot |Z_2-T_1|^2 \cdot
|W_1-T_1|^2 \cdot |Z_1-T_2|^2 \cdot |W_1-T_2|^2 \cdot |W_2-T_2|^2},
\end{multline*}
where we write $|Z-W|^2$ for $N(Z-W)$ in order to fit the formula on page.
The roles of variables $T_1$ and $T_2$ are symmetric, so we shall assume that
$r_1 > r_2 > 0$.
The purpose of the factor $N(Z_1-W_1)$ in (\ref{l2-rel}) is to give $l^{(2)}$
desired conformal properties (see Lemma \ref{conformal}).

These are the only ladder integrals that we consider in this paper.
In general, one obtains the integral from the ladder diagram by building
a rational function by writing a factor
$$
\begin{cases}
N(Y_i-Y_j)^{-1} &
\text{if there is a solid edge joining variables $Y_i$ and $Y_j$};  \\
N(Y_i-Y_j) & \text{if there is a dashed edge joining variables $Y_i$ and $Y_j$},
\end{cases}
$$
then integrating over the solid vertices.
If desired, by Corollary 90 in \cite{FL1} the integrals over various $U(2)_R$
can be replaced by integrals over the Minkowski space $\BB M$ via an
appropriate ``Cayley transform''.
The ladder diagrams are obtained by starting with the one-loop ladder diagram
(Figure \ref{12ladder}) and adding the so-called ``slingshots'',
as explained in \cite{DHSS}.

From Lemmas \ref{Z-W}, \ref{Jacobian_lemma} and the fact that the integrand
is a closed differential form we immediately obtain:

\begin{lem}  \label{conformal}
For each $h = \bigl(\begin{smallmatrix} a' & b' \\
c' & d' \end{smallmatrix}\bigr) \in GL(2,\HC)$
sufficiently close to the identity we have:
\begin{multline*}
l^{(1)}(\tilde Z_1, \tilde Z_2; \tilde W_1, \tilde W_2) \\
= N(a'-Z_1c') \cdot N(cZ_2+d) \cdot N(cW_1+d) \cdot N(a'-W_2c') \cdot
l^{(1)}(Z_1,Z_2;W_1,W_2),
\end{multline*}
\begin{multline*}
\tilde l^{(2)}(\tilde Z_1, \tilde Z_2; \tilde W_1, \tilde W_2)
= N(cZ_1+d) \cdot N(a'-Z_1c') \cdot N(cZ_2+d) \\
\cdot N(cW_1+d) \cdot N(a'-W_1c') \cdot N(a'-W_2c') \cdot
\tilde l^{(2)}(Z_1,Z_2;W_1,W_2),
\end{multline*}
\begin{multline*}
l^{(2)}(\tilde Z_1, \tilde Z_2; \tilde W_1, \tilde W_2) \\
= N(a'-Z_1c') \cdot N(cZ_2+d) \cdot N(cW_1+d) \cdot N(a'-W_2c') \cdot
l^{(2)}(Z_1,Z_2;W_1,W_2),
\end{multline*}
where $h^{-1} = \bigl(\begin{smallmatrix} a & b \\
c & d \end{smallmatrix}\bigr)$,
$\tilde Z_i = (aZ_i+b)(cZ_i+d)^{-1}$ and $\tilde W_i = (aW_i+b)(cW_i+d)^{-1}$,
$i=1,2$.
\end{lem}

\subsection{Integral Operators Corresponding to the Ladder Diagrams}

Using bilinear pairings (\ref{H-pairing}) and (\ref{pairing}) we obtain
integral operators
\begin{align*}
L^{(1)} &: (\pi^0_l, {\cal H}^+) \otimes (\pi^0_r, {\cal H}^+) \to
(\pi^0_l, {\cal H}^+) \otimes (\pi^0_r, {\cal H}^+),  \\
\tilde L^{(2)} &: (\rho_1, \Zh) \otimes (\pi^0_r, {\cal H}^+) \to
(\rho_1, \Zh^+) \otimes (\pi^0_r, {\cal H}^+),  \\
\mathring{L}^{(2)} &: (\varpi_2, \Zh) \otimes (\pi^0_r, {\cal H}^+) \to
(\pi^0_l, \Zh^+) \otimes (\pi^0_r, {\cal H}^+)
\end{align*}
that have  $l^{(1)}$, $\tilde l^{(2)}$ and $l^{(2)}$ as their kernels:
\begin{multline*}
L^{(1)} (\phi_1 \otimes \phi_2)(W_1,W_2) \\
= \frac1{(2\pi^2)^2} \iint_{\genfrac{}{}{0pt}{}{Z_1 \in S^3_{R_1}}{Z_2 \in S^3_{R_2}}}
l^{(1)}(Z_1,Z_2;W_1,W_2) \cdot (\degt_{Z_1} \phi_1)(Z_1)
\cdot (\degt_{Z_2} \phi_2)(Z_2) \,\frac{dS_1\,dS_2}{R_1R_2},
\end{multline*}
$$
\tilde L^{(2)} (f \otimes \phi)(W_1,W_2)
= \frac{i}{4\pi^5} \iint_{\genfrac{}{}{0pt}{}{Z_1 \in U(2)_{R_1}}{Z_2 \in S^3_{R_2}}}
\tilde l^{(2)}(Z_1,Z_2;W_1,W_2) \cdot f(Z_1) \cdot (\degt_{Z_2} \phi)(Z_2)
\,dV_1\,\frac{dS_2}{R_2},
$$
$$
\mathring{L}^{(2)} (f \otimes \phi)(W_1,W_2)
= \frac{i}{4\pi^5} \iint_{\genfrac{}{}{0pt}{}{Z_1 \in U(2)_{R_1}}{Z_2 \in S^3_{R_2}}}
l^{(2)}(Z_1,Z_2;W_1,W_2) \cdot f(Z_1) \cdot (\degt_{Z_2} \phi)(Z_2)
\,dV_1\,\frac{dS_2}{R_2},
$$
where $\phi, \phi_1, \phi_2 \in {\cal H}^+$, $f \in \Zh$.
In the case of $L^{(1)}$ we require $R_1, R_2 > 1$, $W_1,W_2 \in \BB D^+$.
In the cases of $\tilde L^{(2)}$ and $\mathring{L}^{(2)}$ we require
$R_1, R_2 > r_1$, $W_1, W_2 \in \BB D^+_{r_2}$.
It follows from the $\mathfrak{gl}(2,\HC)$-invariance of the bilinear
pairings (\ref{H-pairing}), (\ref{pairing}), (\ref{pairing2}) and
Lemma \ref{conformal} that these three integral operators are
$\mathfrak{gl}(2,\HC)$-equivariant.

\begin{rem}
Strictly speaking, we need to show that the functions
$$
L^{(1)} (\phi_1 \otimes \phi_2)(W_1,W_2), \qquad
\tilde L^{(2)} (f \otimes \phi)(W_1,W_2) \qquad \text{and} \qquad
\mathring{L}^{(2)} (f \otimes \phi)(W_1,W_2)
$$
are {\em polynomials} in $W_1$ and $W_2$ as opposed to, say, smooth functions.
This will be done later.
\end{rem}

Finally, we define an integral operator
$$
L^{(2)} : {\cal H}^+ \otimes {\cal H}^+ \to \Zh^+ \otimes {\cal H}^+
$$
using a bilinear pairing (\ref{H-pairing2}) that also has $l^{(2)}$
as its kernel:
\begin{multline*}
L^{(2)} (\phi_1 \otimes \phi_2)(W_1,W_2) \\
= \frac{i}{4\pi^5} \iint_{\genfrac{}{}{0pt}{}{Z_1 \in U(2)_{R_1}}{Z_2 \in S^3_{R_2}}}
l^{(2)}(Z_1,Z_2;W_1,W_2) \cdot (\degt_{Z_1} \phi_1)(Z_1)
\cdot (\degt_{Z_2} \phi_2)(Z_2) \,\frac{dV_1}{N(Z_1)} \frac{dS_2}{R_2},
\end{multline*}
where $\phi_1, \phi_2 \in {\cal H}^+$, $W_1, W_2 \in \BB D^+_{r_2}$,
$R_1, R_2 > r_1$, as before.

\begin{rem}  \label{discussion}
At this point it is easy to see that $L^{(2)} (\phi_1 \otimes \phi_2)(W_1,W_2)$
is harmonic with respect to the $W_2$ variable, but it is not at all obvious
whether it is harmonic with respect to the $W_1$ variable or not.
Since $l^{(2)}(Z_1,Z_2;W_1,W_2)$ may or may not be harmonic with respect to
the $Z_1$ variable, it is also not clear if the operator $L^{(2)}$ is
$\mathfrak{gl}(2,\HC)$-equivariant.
However, we will see later (Theorem \ref{main-thm}) that
$L^{(2)} (\phi_1 \otimes \phi_2)(W_1,W_2)$ is indeed harmonic with respect to
the $W_1$ variable and that we have a $\mathfrak{gl}(2,\HC)$-equivariant map
(\ref{L2}).
\begin{equation}  \label{L2}
L^{(2)} : (\pi^0_l, {\cal H}^+) \otimes (\pi^0_r, {\cal H}^+) \to
(\pi^0_l, {\cal H}^+) \otimes (\pi^0_r, {\cal H}^+).
\end{equation}
\end{rem}

One of the central results of \cite{FL1} was to show that the operator
$L^{(1)}$ corresponding to the one-loop ladder diagram is the
$\mathfrak{gl}(2,\HC)$-equivariant projection onto the first irreducible
component (see (\ref{decomp-intro}), (\ref{tensor-decomp}))
\begin{equation*}
L^{(1)}: (\pi^0_l, {\cal H}^+) \otimes (\pi^0_r, {\cal H}^+)
\twoheadrightarrow (\rho_1,\Zh^+) \hookrightarrow
(\pi^0_l, {\cal H}^+) \otimes (\pi^0_r, {\cal H}^+)
\end{equation*}
(the multiplication map followed by the embedding) such that
$L^{(1)}(1 \otimes 1) = 1 \otimes 1$.
The goal of this article is to understand the map (\ref{L2}).

We conclude this subsection by observing some relations between operators
$\tilde L^{(2)}$, $\mathring{L}^{(2)}$ and $L^{(2)}$.
From (\ref{l2-rel}) we obtain the following relation:
\begin{equation}  \label{LL1}
\mathring{L}^{(2)} = \tilde L^{(2)} \circ N(Z_1-W_1),
\end{equation}
where (by abuse of notation) $N(Z_1-W_1)$ denotes multiplication by $N(Z_1-W_1)$.
We also have:
\begin{equation}  \label{LL2}
L^{(2)} = \mathring{L}^{(2)} \circ \bigl( N(Z_1)^{-1} \cdot \degt_{Z_1} \bigr).
\end{equation}

\section{Equivariant Maps $\tilde L^{(2)}$ and
$(\rho_1,\Zh) \to (\pi_l^0, {\cal H}) \otimes (\pi_r^0,{\cal H})$}  \label{I_R-section}

\subsection{Equivariant Maps
$(\rho_1,\Zh) \to (\pi_l^0, {\cal H}) \otimes (\pi_r^0,{\cal H})$}

A tensor product $(\pi_l^0, {\cal H}^+) \otimes (\pi_r^0,{\cal H}^+)$ of
representations of $\mathfrak{gl}(2,\HC)$ decomposes into a direct sum
of irreducible subrepresentations, one of which is $(\rho_1,\Zh^+)$.
This decomposition is stated precisely in equation (\ref{tensor-decomp}).
The irreducible component $(\rho_1,\Zh^+)$ has multiplicity one
and is generated by $1 \otimes 1 \in {\cal H}^+ \otimes {\cal H}^+$.
Thus we have a $\mathfrak{gl}(2,\HC)$-equivariant map
\begin{equation*}
I: (\rho_1,\Zh^+) \hookrightarrow
(\pi_l^0, {\cal H}^+) \otimes (\pi_r^0,{\cal H}^+),
\end{equation*}
which is unique up to multiplication by a scalar.
This scalar can be pinned down by a requirement $I(1) = 1 \otimes 1$.

We consider a map
\begin{equation}  \label{fork}
\Zh \ni f(Z) \quad \mapsto \quad (I_R f)(W_1,W_2) =
\frac i{2\pi^3} \int_{Z \in U(2)_R} \frac{f(Z) \,dV}{N(Z-W_1) \cdot N(Z-W_2)}
\quad \in \B{{\cal H} \otimes {\cal H}},
\end{equation}
where $\B{{\cal H} \otimes {\cal H}}$ denotes the Hilbert space obtained by
completing ${\cal H} \otimes {\cal H}$ with respect to the unitary structure
coming from the tensor product of unitary representations
$(\pi^0_l, {\cal H})$ and $(\pi^0_r, {\cal H})$.
If $W_1, W_2 \in \BB D^+_R$ or $W_1, W_2 \in \BB D^-_R$, the integrand has
no singularities and the result is a holomorphic function in two variables
$W_1, W_2$ which is harmonic in each variable separately.
Recall that $M$ denotes the multiplication map (\ref{M}).




\begin{thm}[Theorem 12 in \cite{FL3}]  \label{embedding-thm}
The map $f(Z) \mapsto (I_R f)(W_1,W_2)$ has the following properties:
\begin{enumerate}
\item
If $W_1, W_2 \in \BB D^+_R$, then $I_R: \Zh \to {\cal H}^+ \otimes {\cal H}^+$,
$$
M \circ (I_R f)(W_1,W_2) = f \quad \text{if $f \in \Zh^+$}
\qquad \text{and} \qquad
(I_R f)(W_1,W_2) = 0 \quad \text{if $f \in \Zh^-_1 \oplus \Zh^0$}.
$$
The restriction of $I_R$ to $\Zh^+$ coincides with the map $I$.
\item
If $W_1, W_2 \in \BB D^-_R$, then $I_R: \Zh \to {\cal H}^- \otimes {\cal H}^-$,
$$
M \circ (I_R f)(W_1,W_2) = f \quad \text{if $f \in \Zh^-_1$}
\qquad \text{and} \qquad
(I_R f)(W_1,W_2) = 0 \quad \text{if $f \in \Zh^0 \oplus \Zh^+$}.
$$
\end{enumerate}
\end{thm}

We finish this subsection with a lemma that will be used in
our computation of the map $\tilde L^{(2)}$ on
$(\rho_1,\Zh^+) \otimes (\pi_r^0, {\cal H}^+)$.

\begin{lem}  \label{z^p}
Let $p=1,2,3,\dots$, and let $z_{ij}=z_{11}$, $z_{12}$, $z_{21}$ or $z_{22}$. Then
$$
\bigl( I(z_{ij})^p \bigr)(W,W') =
\frac1{p+1} \sum_{k=0}^p (w_{ij})^k \cdot (w'_{ij})^{p-k}.
$$
\end{lem}

\begin{proof}
By direct calculation,
\begin{center}
\begin{tabular}{rcr}
$t^l_{-l\,\underline{-l}}(Z) = (z_{11})^{2l}$, & \qquad &
$t^l_{-l\,\underline{l}}(Z) = (z_{12})^{2l}$, \\
$t^l_{l\,\underline{-l}}(Z) = (z_{21})^{2l}$, & \qquad &
$t^l_{l\,\underline{l}}(Z) = (z_{22})^{2l}$.
\end{tabular}
\end{center}
We give the proof for the case $z_{ij}=z_{11}$, the other cases are similar.
Applying the matrix coefficient expansion (\ref{1/N-expansion}), we obtain:
\begin{multline*}
\bigl( I_R (z_{11})^p \bigr)(W,W')
= \frac i{2\pi^3} \int_{Z \in U(2)_R} \frac{(z_{11})^p \,dV}{N(Z-W) \cdot N(Z-W')} \\
= \biggl\langle \frac1{N(Z-W) \cdot N(Z-W')},
t^{p/2}_{-p/2\,\underline{-p/2}}(Z) \biggr\rangle_Z \\
= \sum_{l,m,n,l',m',n'} t^l_{n \, \underline{m}}(W) \cdot t^{l'}_{n' \, \underline{m'}}(W')
\cdot \bigl\langle N(Z)^{-2} \cdot t^l_{m \, \underline{n}}(Z^{-1}) \cdot
t^{l'}_{m' \, \underline{n'}}(Z^{-1}), t^{p/2}_{-p/2\,\underline{-p/2}}(Z) \bigr\rangle_Z.
\end{multline*}
By the orthogonality relations (\ref{orthogonality}),
$$
\bigl\langle N(Z)^{-2} \cdot t^l_{m \, \underline{n}}(Z^{-1}) \cdot
t^{l'}_{m' \, \underline{n'}}(Z^{-1}), t^{p/2}_{-p/2\,\underline{-p/2}}(Z) \bigr\rangle
=0 \qquad \text{if $l+l' \ne p/2$}
$$
and
$$
\bigl\langle N(Z)^{-2} \cdot t^l_{-l \, \underline{-l}}(Z^{-1}) \cdot
t^{l'}_{-l' \, \underline{-l'}}(Z^{-1}), t^{p/2}_{-p/2\,\underline{-p/2}}(Z) \bigr\rangle
= \frac1{p+1} \qquad \text{if $l+l' = p/2$}.
$$
Finally, we need to show that if $l+l'= p/2$ and
$(m,n,m',n') \ne (-l,-l,-l',-l')$, then
\begin{equation}  \label{zero}
\bigl\langle N(Z)^{-2} \cdot t^l_{m \, \underline{n}}(Z^{-1}) \cdot
t^{l'}_{m' \, \underline{n'}}(Z^{-1}), t^{p/2}_{-p/2\,\underline{-p/2}}(Z) \bigr\rangle =0.
\end{equation}
Indeed, by (\ref{t}), each $t^a_{b \, \underline{c}}(Z)$ is a linear combination of
monomials
$$
(z_{11})^{\alpha_{11}} (z_{12})^{\alpha_{12}} (z_{21})^{\alpha_{21}}
(z_{22})^{2a-\alpha_{11}-\alpha_{12}-\alpha_{21}},
$$
and if $(b,c) \ne (-a,-a)$, then
$t^a_{b \, \underline{c}}(Z)$ does not contain the monomial $(z_{11})^{2a}$.
Hence the product $t^l_{m \, \underline{n}}(Z) \cdot t^{l'}_{m' \, \underline{n'}}(Z)$
does not contain the monomial $(z_{11})^p$, and the expansion of
$t^l_{m \, \underline{n}}(Z) \cdot t^{l'}_{m' \, \underline{n'}}(Z)$
into basis functions (\ref{Zh-basis}) does not contain the term 
$t^{p/2}_{-p/2\,\underline{-p/2}}(Z)$. Thus (\ref{zero}) follows.
Therefore,
$$
\bigl( I_R (z_{ij})^p \bigr)(W,W') = \frac1{p+1} \sum_{k=0}^p
t^{\frac{k}2}_{-\frac{k}2\,\underline{-\frac{k}2}}(W) \cdot
t^{\frac{p-k}2}_{-\frac{p-k}2\,\underline{-\frac{p-k}2}}(W').
$$
\end{proof}

For future use, we state the following consequence of this proof:

\begin{cor}
Let $k \ge 0$. We have the following orthogonality relations:
$$
\bigl\langle N(Z)^{-2-k} \cdot t^l_{m \, \underline{n}}(Z^{-1}) \cdot
t^{l'}_{m' \, \underline{n'}}(Z^{-1}), t^{p/2}_{-p/2\,\underline{-p/2}}(Z) \bigr\rangle =
\begin{cases}
\frac1{p+1} & \text{if $k=0$, $l+l'=p/2$,} \\
& \text{$m=n=-l$ and $m'=n'=-l'$}; \\
0 & \text{otherwise}.
\end{cases}
$$
\end{cor}

\subsection{Some Irreducible Components of
$(\rho_1,\Zh^+) \otimes (\pi_r^0, {\cal H}^+)$}  \label{irred-comp-subsection}

In this subsection we describe some irreducible components of
$(\rho_1,\Zh^+) \otimes (\pi_r^0, {\cal H}^+)$.
Decompositions of tensor products of similar representations of $SU(n,n)$
(instead of just $SU(2,2)$) were studied, for example, in \cite{J1,J2,JV}.
But we could not find the decomposition of this particular tensor product
in the literature.

We denote by $\BB C^{n \times n}$ the space of complex $n \times n$ matrices.
Then $\widetilde{\Zh} \otimes \BB C^{n \times n}$ is the space of holomorphic
functions on $\HC$ (possibly with singularities) with values in
$\BB C^{n \times n}$.
We let parameters $m,n=1,2,3,\dots$ and consider the following actions of
$GL(2,\HC)$ on $\widetilde{\Zh} \otimes \BB C^{n \times n}$:
\begin{equation}  \label{pi_mn-action}
\varpi_m^n(h): \: F(Z) \quad \mapsto \quad \bigl( \varpi_m^n(h)F \bigr)(Z) =
\frac {\tau_{\frac{n-1}2}(cZ+d)^{-1}}{N(cZ+d)^m} \cdot
F \bigl( (aZ+b)(cZ+d)^{-1} \bigr) \cdot
\frac {\tau_{\frac{n-1}2}(a'-Zc')^{-1}}{N(a'-Zc')},
\end{equation}
where
$h = \bigl(\begin{smallmatrix} a' & b' \\ c' & d' \end{smallmatrix}\bigr)
\in GL(2,\HC)$,
$h^{-1} = \bigl(\begin{smallmatrix} a & b \\ c & d \end{smallmatrix}\bigr)$,
expressions $cZ+d$ and $a'-Zc'$ are regarded as elements of $\HC^{\times}$
and $\tau_l: \HC^{\times} \to\operatorname{Aut}(\BB C^{2l+1}) \subset
\BB C^{(2l+1) \times (2l+1)}$
is the irreducible $(2l+1)$-dimensional representation of $\HC^{\times}$
described in Subsection \ref{matrix-coeff-subsection},
$l=0,\frac12,1,\frac32, \dots$.

For $n=1$, $\tau_0 \equiv 1$ and $\varpi_m^1 \equiv \varpi_m$.
On the other hand, if $m=1$, then $\varpi_1^n \equiv \rho_n$,
where the action $\rho_n$ is described by equation (60) in \cite{FL1}.
Differentiating the $\varpi_m^n$-action, we obtain an action of
$\mathfrak{gl}(2,\HC)$ which preserves $\Zh \otimes \BB C^{n \times n}$
and $\Zh^+ \otimes \BB C^{n \times n}$.
As a special case of Proposition 4.7 in \cite{JV}
(see also the discussion preceding the proposition and references therein),
we have:

\begin{thm} \label{JV-thm}
The representations $(\varpi_m^n, \Zh^+ \otimes \BB C^{n \times n})$,
$m,n=1,2,3,\dots$, of $\mathfrak{sl}(2,\HC)$ are irreducible.
They possess inner products which make them unitary
representations of the real form $\mathfrak{su}(2,2)$
of $\mathfrak{sl}(2,\HC)$.
\end{thm}

According to \cite{JV}, we have the following decomposition of a tensor
product $(\pi^0_l, {\cal H}^+) \otimes (\pi^0_r, {\cal H}^+)$ into irreducible
subrepresentations of $\mathfrak{gl}(2,\HC)$:
\begin{equation}  \label{tensor-decomp}
(\pi^0_l, {\cal H}^+) \otimes (\pi^0_r, {\cal H}^+) \simeq
\bigoplus_{n=1}^{\infty} (\varpi_1^n,\Zh^+\otimes \BB C^{n \times n})
= \bigoplus_{n=1}^{\infty} (\rho_n,\Zh^+\otimes \BB C^{n \times n})
\end{equation}
(see also Subsection 5.1 in \cite{FL1}).
We outline the proof of this statement.
First of all, by Lemma \ref{Z-W}, the tensor product
$(\pi^0_l, {\cal H}^+) \otimes (\pi^0_r, {\cal H}^+)$ contains each
$(\varpi_1^n,\Zh^+\otimes \BB C^{n \times n})$ with 
$$
(\varpi_1^1,\Zh^+) \quad \text{generated by} \quad 1 \otimes 1
$$
and
$$
(\varpi_1^n,\Zh^+\otimes \BB C^{n \times n}) \quad \text{generated by} \quad
(z_{ij}-z'_{ij})^{n-1} , \qquad n \ge 2.
$$
Then one checks that the direct sum
$\bigoplus_{n=1}^{\infty} (\varpi_1^n,\Zh^+\otimes \BB C^{n \times n})$
exhausts all of $(\pi^0_l, {\cal H}^+) \otimes (\pi^0_r, {\cal H}^+)$
by comparing the two sides as representations of $U(2) \times U(2)$ or
$\mathfrak{u}(2) \times \mathfrak{u}(2)$.

Similarly, define subrepresentations $(\rho_1 \otimes \pi_r^0,\mathfrak{V}_n)$
of $(\rho_1,\Zh^+) \otimes (\pi_r^0, {\cal H}^+)$ as
$$
\mathfrak{V}_n =
\text{ the smallest $\mathfrak{gl}(2,\HC)$-invariant subspace containing }
\begin{cases}
1 \otimes 1 & \text{if $n=1$}; \\
(z_{11}-z'_{11})^{n-1} & \text{if $n \ge 2$}.
\end{cases}
$$
Then each $\mathfrak{V}_n$ can be $\mathfrak{gl}(2,\HC)$-equivariantly
mapped onto $(\varpi_2^n,\Zh^+\otimes \BB C^{n \times n})$, $n \ge 1$.
(It is possible that some $\mathfrak{V}_n$'s are actually isomorphic to
$(\varpi_2^n,\Zh^+\otimes \BB C^{n \times n})$.)
Thus $(\rho_1,\Zh^+) \otimes (\pi_r^0, {\cal H}^+)$ contains each
$(\varpi_2^n,\Zh^+\otimes \BB C^{n \times n})$, $n \ge 1$, among its irreducible
components.
Note that $\Zh^+ \otimes {\cal H}^+$ must contain more irreducible
components in addition to those, which can be seen by, for example, comparing
the two sides as representations of $U(2) \times U(2)$ or
$\mathfrak{u}(2) \times \mathfrak{u}(2)$.

We introduce another subrepresentation
$(\Zh^+ \otimes {\cal H}^+)_1 \subset \Zh^+ \otimes {\cal H}^+$ as
$$
(\Zh^+ \otimes {\cal H}^+)_1 =
\mathfrak{V}_1 + \mathfrak{V}_2 + \dots + \mathfrak{V}_n + \dots.
$$



\subsection{The Effect of $\tilde L^{(2)}$ on
$(\Zh^+ \otimes {\cal H}^+)_1 \subset \Zh^+ \otimes {\cal H}^+$}

In this subsection we compute the effect of $\tilde L^{(2)}$ on
$(\Zh^+ \otimes {\cal H}^+)_1 \subset \Zh^+ \otimes {\cal H}^+$.
These calculations will be used to compute the map
$L^{(2)}$ on $(\pi^0_l, {\cal H}^+) \otimes (\pi^0_r, {\cal H}^+)$.

\begin{figure}
\begin{center}
\begin{subfigure}{0.3\textwidth}
\centering
\setlength{\unitlength}{1mm}
\begin{picture}(20,38)
\put(2,4){\line(0,1){30}}
\put(17,4){\line(0,1){30}}
\put(2,4){\line(3,2){15}}
\put(2,24){\line(3,2){15}}
\put(2,24){\line(3,-2){15}}

\put(2,24){\circle*{2}}
\put(17,14){\circle*{2}}

\put(0,35){$Z_2$}
\put(16,35){$Z_1$}
\put(16,0){$W_2$}
\put(0,0){$W_1$}
\put(20,13){$T_2$}
\put(-4,23){$T_1$}
\end{picture}
\end{subfigure}
$\rightsquigarrow$
\begin{subfigure}{0.3\textwidth}
\centering
\setlength{\unitlength}{1mm}
\begin{picture}(20,48)
\multiput(2,4)(0,15){3}{\line(0,1){10}}
\multiput(17,4)(0,15){3}{\line(0,1){10}}
\put(2,4){\line(3,2){15}}
\put(2,34){\line(3,2){15}}
\put(2,29){\line(3,-2){15}}

\multiput(2,15)(0,2){2}{\line(0,1){1}}
\multiput(17,15)(0,2){2}{\line(0,1){1}}
\multiput(2,30)(0,2){2}{\line(0,1){1}}
\multiput(17,30)(0,2){2}{\line(0,1){1}}

\put(0,45){$Z_2$}
\put(16,45){$Z_1$}
\put(16,0){$W_2$}
\put(0,0){$W_1$}
\end{picture}
\end{subfigure}
\end{center}
\caption{Decomposition of the diagram for $\tilde l^{(2)}$ into
three zig-zag diagrams.}
\label{zig-zag-decomposition}
\end{figure}
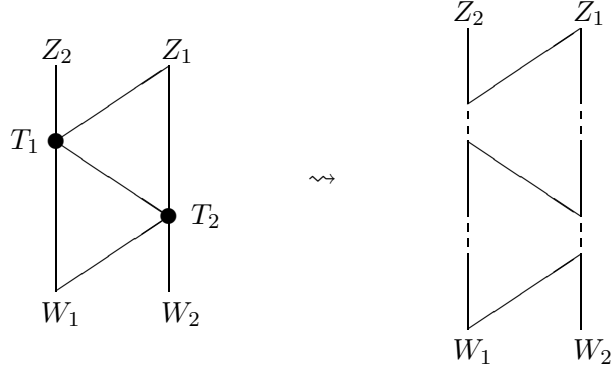

\begin{figure}
\begin{center}
\setlength{\unitlength}{1mm}
\begin{picture}(20,18)
\put(2,4){\line(0,1){10}}
\put(17,4){\line(0,1){10}}
\put(2,4){\line(3,2){15}}

\put(0,15){$Z'$}
\put(16,15){$Z$}
\put(16,0){$T'$}
\put(0,0){$T$}
\end{picture}
\end{center}
\caption{Zig-zag diagram.}
\label{zig-zag}
\end{figure}
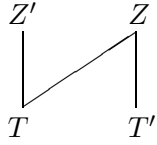

Although $\tilde l^{(2)}$ is not a ladder integral, we can think of it
as represented by the diagram on the left side of
Figure \ref{zig-zag-decomposition}.
(It is the two-loop ladder diagram with dashed line deleted.)
As shown in Figure \ref{zig-zag-decomposition}, we can break the diagram
into three ``zig-zags''.
To each zig-zag diagram as in Figure \ref{zig-zag},
we associate a function
$$
\lambda(Z,Z';T,T') = \frac1{N(Z'-T) \cdot N(T-Z) \cdot N(Z-T')}.
$$
From Lemma \ref{Z-W} we immediately obtain the following
conformal property of this function:

\begin{lem}
If $h = \bigl(\begin{smallmatrix} a' & b' \\ c' & d' \end{smallmatrix}\bigr)
\in GL(2,\HC)$,
$h^{-1} = \bigl(\begin{smallmatrix} a & b \\ c & d \end{smallmatrix}\bigr)$,
denote by $\tilde Z = (aZ+b)(cZ+d)^{-1}$ and define
$\tilde Z'$, $\tilde T$, $\tilde T'$ similarly. Then
\begin{multline*}
\lambda(\tilde Z, \tilde Z'; \tilde T, \tilde T')
= N(cZ+d) \cdot N(a'-Zc') \cdot N(cZ'+d) \\
\cdot N(cT+d) \cdot N(a'-Tc') \cdot N(a'-T'c') \cdot \lambda(Z,Z';T,T').
\end{multline*}
\end{lem}

Corresponding to this function $\lambda$, we have an integral operator
$\Lambda$ on $(\rho_1,\Zh) \otimes (\pi_r^0, {\cal H}^+)$ defined by
$$
\Lambda (f \otimes \phi)(T,T') \\
= \frac{i}{16\pi^5} \iint_{\genfrac{}{}{0pt}{}{Z \in U(2)_{R}}{Z' \in S^3_{R'}}}
\lambda(Z,Z';T,T') \cdot f(Z) \cdot (\degt_{Z'} \phi)(Z')
\,dV_Z\,\frac{dS_{Z'}}{R'},
$$
where $f \in \Zh$, $\phi \in {\cal H}^+$, $T, T' \in \BB D^+_r$,
$R, R'>0$ and $r=\min\{R,R'\}$.
Since the bilinear pairings (\ref{H-pairing}) and (\ref{pairing}) are
$\g{gl}(2,\HC)$-equivariant, so is
$$
\Lambda: (\rho_1,\Zh) \otimes (\pi_r^0, {\cal H}^+)
\to (\rho_1,\Zh) \otimes (\pi_r^0, {\cal H}^+).
$$
Then the map $\tilde L^{(2)} : (\rho_1,\Zh) \otimes (\pi_r^0, {\cal H}^+)
\to (\rho_1,\Zh) \otimes (\pi_r^0, {\cal H}^+)$ is a composition of three
copies of $\Lambda$:
\begin{equation}  \label{Lambda-cubed}
\tilde L^{(2)} = \Lambda \circ \Lambda \circ \Lambda.
\end{equation}

\begin{prop}  \label{Lambda-prop}
The operator $\Lambda$ annihilates $(\Zh^-_1 \oplus \Zh^0) \otimes {\cal H}^+$,
and its image lies in $\Zh^+ \otimes {\cal H}^+$.
If $x \in (\Zh^+ \otimes {\cal H}^+)_1$ belongs to
$\mathfrak{V}_n$ -- the subrepresentation of $\Zh^+ \otimes {\cal H}^+$
generated by $(z_{11}-z'_{11})^{n-1}$ -- then
$$
\Lambda(x) = \lambda_n x, \qquad \text{where} \qquad
\lambda_n =
\begin{cases}
1 & \text{if $n=1$;} \\
\frac{(-1)^{n+1}}n & \text{if $n \ge 2$.}
\end{cases}
$$
\end{prop}

\begin{proof}
By Theorem \ref{Poisson} and Theorem \ref{embedding-thm}, the operator
$\Lambda$ is a composition of the canonical isomorphism switching the
components
$$
(\rho_1,\Zh) \otimes (\pi_r^0, {\cal H}^+) \simeq
(\pi_r^0, {\cal H}^+) \otimes (\rho_1,\Zh),
\qquad f \otimes \phi \mapsto \phi \otimes f,
$$
followed by the projection
$$
Id_{{\cal H}^+} \otimes Proj : (\pi_r^0, {\cal H}^+) \otimes (\rho_1,\Zh)
\twoheadrightarrow (\pi_r^0, {\cal H}^+) \otimes (\rho_1,\Zh^+),
$$
where $Proj: \Zh = \Zh^-_1 \oplus \Zh^0 \oplus \Zh^+ \twoheadrightarrow \Zh^+$
is the projection, followed by the inclusion
$$
I \otimes Id_{{\cal H}^+} : (\pi_r^0, {\cal H}^+) \otimes (\rho_1,\Zh^+)
\hookrightarrow (\pi_r^0, {\cal H}^+) \otimes (\pi_l^0, {\cal H}^+)
\otimes (\pi_r^0, {\cal H}^+)
$$
and followed by the multiplication map
$$
M \otimes Id_{{\cal H}^+} : (\pi_r^0, {\cal H}^+) \otimes (\pi_l^0, {\cal H}^+)
\otimes (\pi_r^0, {\cal H}^+) \to (\rho_1,\Zh^+) \otimes (\pi_r^0, {\cal H}^+)
$$
defined on pure tensors by
$$
\phi_1(Z_1) \otimes \phi_2(Z_2) \otimes \phi_3(Z_3) \mapsto
(\phi_1 \cdot \phi_2)(T) \otimes \phi_3(T').
$$
In particular, the operator $\Lambda$ annihilates
$(\Zh^-_1 \oplus \Zh^0) \otimes {\cal H}^+$,
and its image lies in $\Zh^+ \otimes {\cal H}^+$.


Next we compute the action of $\Lambda$ on the generators of $\mathfrak{V}_n$.

\begin{lem}  \label{Lambda-lem}
We have: $\Lambda(1 \otimes 1) = 1 \otimes 1$ and
$$
\Lambda: (z_{11}-z'_{11})^n \: \mapsto \: \frac{(-1)^n}{n+1}(t_{11}-t'_{11})^n,
\qquad n \ge 1.
$$
\end{lem}

\begin{proof}
It is clear that $\Lambda(1 \otimes 1) = 1 \otimes 1$, so let us assume
$n \ge 1$.
From the description of $\Lambda$ as a composition of four mappings
and Lemma \ref{z^p}, it follows that $\Lambda$ maps
$$
(z_{11}-z'_{11})^n = \sum_{k=0}^n (-1)^{k} \begin{pmatrix} n \\ k \end{pmatrix}
(z_{11})^{n-k} (z'_{11})^k
$$
into
\begin{multline*}
\sum_{k=0}^n \sum_{p=0}^{n-k} \frac{(-1)^k}{n-k+1}
\begin{pmatrix} n \\ k \end{pmatrix} (t_{11})^{k+p} (t'_{11})^{n-k-p}
= \sum_{k=0}^n \sum_{p=0}^{n-k} \frac{(-1)^k}{n+1}
\begin{pmatrix} n+1 \\ k \end{pmatrix} (t_{11})^{k+p} (t'_{11})^{n-k-p}  \\
= \frac1{n+1}\sum_{r=0}^n \sum_{k=0}^r (-1)^k
\begin{pmatrix} n+1 \\ k \end{pmatrix} (t_{11})^r (t'_{11})^{n-r}  \\
= \frac1{n+1}\sum_{r=0}^n (-1)^r \begin{pmatrix} n \\ r \end{pmatrix}
(t_{11})^r (t'_{11})^{n-r}
= \frac{(-1)^n}{n+1} (t_{11}-t'_{11})^n,
\end{multline*}
where we used an identity
$$
\sum_{k=0}^r (-1)^k \begin{pmatrix} n+1 \\ k \end{pmatrix}
= (-1)^r \begin{pmatrix} n \\ r \end{pmatrix}
$$
which can be easily proved by induction (see formula 0.15(4) in \cite{GR}).
\end{proof}

Since, $\Lambda$ is $\mathfrak{gl}(2,\HC)$-equivariant and maps the generator
of each $\mathfrak{V}_n$ into $\lambda_n$ multiple of itself, $\Lambda$ must
act by multiplication by $\lambda_n$ on the whole $\mathfrak{V}_n$.
\end{proof}

As immediate consequences of this proposition and (\ref{Lambda-cubed})
we obtain:

\begin{cor}
The subrepresentation $(\Zh^+ \otimes {\cal H}^+)_1$ is a direct sum of
$\mathfrak{V}_n$'s:
$$
(\Zh^+ \otimes {\cal H}^+)_1 = \bigoplus_{n=1}^{\infty} \mathfrak{V}_n.
$$
\end{cor}

\begin{thm}  \label{main2-thm}
The operator $\tilde L^{(2)}$ annihilates
$(\Zh^-_1 \oplus \Zh^0) \otimes {\cal H}^+$,
and its image lies in $\Zh^+ \otimes {\cal H}^+$.
If $x \in (\Zh^+ \otimes {\cal H}^+)_1$ belongs to
$\mathfrak{V}_n$ -- the subrepresentation of $\Zh^+ \otimes {\cal H}^+$
generated by $(z_{11}-z'_{11})^{n-1}$ -- then
$$
\tilde L^{(2)}(x) = \tilde\lambda_n x, \qquad \text{where} \qquad
\tilde\lambda_n = \lambda_n^3 =
\begin{cases}
1 & \text{if $n=1$;} \\
\frac{(-1)^{n+1}}{n^3} & \text{if $n \ge 2$.}
\end{cases}
$$
\end{thm}

\section{The Two-Loop Ladder Diagram and
$(\pi^0_l, {\cal H}^+) \otimes (\pi^0_r, {\cal H}^+)$}  \label{main-section}

In this section we combine the results we obtained so far to compute
the effect of the integral operator $L^{(2)}$ on
$(\pi^0_l, {\cal H}^+) \otimes (\pi^0_r, {\cal H}^+)$.
(Recall that $L^{(2)}$ is the operator corresponding to the two-loop ladder
diagram.)

\begin{thm}  \label{main-thm}
The image of the operator $L^{(2)}$ lies in ${\cal H}^+ \otimes {\cal H}^+$,
and the map
\begin{equation}  \label{L2-2}
L^{(2)} : (\pi^0_l, {\cal H}^+) \otimes (\pi^0_r, {\cal H}^+) \to
(\pi^0_l, {\cal H}^+) \otimes (\pi^0_r, {\cal H}^+)
\end{equation}
is $\mathfrak{gl}(2,\HC)$-equivariant.
If $x \in (\pi^0_l, {\cal H}^+) \otimes (\pi^0_r, {\cal H}^+)$
belongs to an irreducible component isomorphic to
$(\rho_n,\Zh^+ \otimes \BB C^{n \times n})$
in the decomposition (\ref{tensor-decomp}), then
$$
L^{(2)}(x) = \mu_n x, \qquad \text{where} \qquad
\mu_n =
\begin{cases}
1 & \text{if $n=1$;} \\
\frac{(-1)^{n+1}}{n(n-1)} & \text{if $n \ge 2$.}
\end{cases}
$$
\end{thm}

\begin{proof}
First, we prove a lemma analogous to Lemma \ref{Lambda-lem}.

\begin{lem}  \label{main-lem}
We have: $L^{(2)} (1 \otimes 1) = 1 \otimes 1$ and
$$
L^{(2)}: (z_{11}-z'_{11})^n \: \mapsto \: \frac{(-1)^n}{n(n+1)}(w_{11}-w'_{11})^n,
\qquad n \ge 1.
$$
\end{lem}

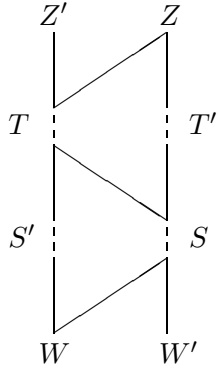
\begin{figure}
\begin{center}
\setlength{\unitlength}{1mm}
\begin{picture}(20,48)
\multiput(2,4)(0,15){3}{\line(0,1){10}}
\multiput(17,4)(0,15){3}{\line(0,1){10}}
\put(2,4){\line(3,2){15}}
\put(2,34){\line(3,2){15}}
\put(2,29){\line(3,-2){15}}

\multiput(2,15)(0,2){2}{\line(0,1){1}}
\multiput(17,15)(0,2){2}{\line(0,1){1}}
\multiput(2,30)(0,2){2}{\line(0,1){1}}
\multiput(17,30)(0,2){2}{\line(0,1){1}}

\put(0,45){$Z'$}
\put(16,45){$Z$}
\put(16,0){$W'$}
\put(0,0){$W$}
\put(20,30.5){$T'$}
\put(-4,30.5){$T$}
\put(20,15.5){$S$}
\put(-4,15.5){$S'$}
\end{picture}
\end{center}
\caption{Variable labeling.}
\label{labeling}
\end{figure}

\begin{proof}
We label the variables in the diagram describing $\tilde l^{(2)}$ as in
Figure \ref{labeling}.
First we compute $L^{(2)} (1 \otimes 1)$.
Using relations (\ref{LL1})-(\ref{LL2}) and the fact that $\tilde L^{(2)}$
annihilates $(\Zh_1^- \oplus \Zh^0) \otimes {\cal H}^+$, we obtain:
\begin{multline*}
L^{(2)} (1 \otimes 1)
= \mathring{L}^{(2)} \bigl( N(Z)^{-1} \cdot \degt_Z (1 \otimes 1) \bigr)
= \mathring{L}^{(2)} \bigl( N(Z)^{-1} \bigr)  \\
= \tilde L^{(2)} \biggl( \frac{N(Z-W)}{N(Z)} \biggr)
= \tilde L^{(2)} \biggl(
1 + \frac{N(W)}{N(Z)} - \frac{\tr(ZW^+)}{N(Z)} \biggr)
= \tilde L^{(2)} (1) = 1 \otimes 1.
\end{multline*}

Next we compute $L^{(2)} \bigl( (z_{11}-z'_{11})^n \bigr)$.
Let us introduce a notation
$$
\alpha_n(Z,Z') = N(Z)^{-1} \cdot \degt_Z \bigl( (z_{11}-z'_{11})^n \bigr),
$$
then, by (\ref{LL1})-(\ref{LL2}),
$$
L^{(2)} \bigl( (z_{11}-z'_{11})^n \bigr)
= \tilde L^{(2)} \bigl( N(Z-W) \cdot \alpha_n(Z,Z') \bigr).
$$
Observe that
$$
N(Z-W) = N(Z) - \tr(ZW^+) + N(W)
$$
and only the terms
$$
\tilde L^{(2)} \bigl( N(Z) \cdot \alpha_n(Z,Z') \bigr)
\qquad \text{and} \qquad
\tilde L^{(2)} \bigl( z_{22}w_{11} \cdot \alpha_n(Z,Z') \bigr)
$$
can potentially be non-zero -- all other terms belong to
$(\Zh_1^- \oplus \Zh^0) \otimes {\cal H}^+$ and thus annihilated by
$\tilde L^{(2)}$.
We have:
\begin{multline*}
N(Z) \cdot \alpha_n(Z,Z') = \degt_Z \bigl( (z_{11}-z'_{11})^n \bigr) \\
= (z_{11}-z'_{11})^n + \sum_{k=0}^n
(-1)^{n-k} \begin{pmatrix} n \\ k \end{pmatrix} k (z_{11})^k (z'_{11})^{n-k}  \\
= (z_{11}-z'_{11})^n + n z_{11} \sum_{l=0}^{n-1}
(-1)^{n-l-1} \begin{pmatrix} n-1 \\ l \end{pmatrix} (z_{11})^l (z'_{11})^{n-l-1} \\
= (z_{11}-z'_{11})^n + n z_{11} (z_{11}-z'_{11})^{n-1} \\
= (n+1)(z_{11}-z'_{11})^n + n z'_{11} (z_{11}-z'_{11})^{n-1}.
\end{multline*}
Then, using (\ref{Lambda-cubed}) and Proposition \ref{Lambda-prop},
\begin{multline}  \label{last1}
\tilde L^{(2)} \bigl( N(Z) \cdot \alpha_n(Z,Z') \bigr) =
\tilde L^{(2)} \bigl( (n+1)(z_{11}-z'_{11})^n + n z'_{11} (z_{11}-z'_{11})^{n-1} \bigr)
\\
= (-1)^n (\Lambda\circ\Lambda)
\bigl( (t_{11}-t'_{11})^n - t_{11}(t_{11}-t'_{11})^{n-1} \bigr)
= (-1)^{n+1} (\Lambda \circ \Lambda) \bigl( t'_{11}(t_{11}-t'_{11})^{n-1} \bigr) \\
= \frac1n \Lambda \bigl( s_{11} (s_{11}-s'_{11})^{n-1} \bigr)
= \frac1n \Lambda \bigl( (s_{11}-s'_{11})^n + s'_{11} (s_{11}-s'_{11})^{n-1} \bigr) \\
= \frac{(-1)^n}{n(n+1)}(w_{11}-w'_{11})^n
+ \frac{(-1)^{n-1}}{n^2}w_{11}(w_{11}-w'_{11})^{n-1}.
\end{multline}

Finally, we compute
$\tilde L^{(2)} \bigl( z_{22}w_{11} \cdot \alpha_n(Z,Z') \bigr)$:
$$
z_{22}w_{11} \cdot \alpha_n(Z,Z')
= \frac{z_{22}w_{11}}{N(Z)} \sum_{k=0}^n (-1)^{n-k}
\begin{pmatrix} n \\ k \end{pmatrix} (k+1) (z_{11})^k (z'_{11})^{n-k}.
$$
Since the terms in $(\Zh_1^- \oplus \Zh^0) \otimes {\cal H}^+$ are
annihilated by $\tilde L^{(2)}$, we can drop them.
By (\ref{Z^+f}), modulo terms in $(\Zh_1^- \oplus \Zh^0) \otimes {\cal H}^+$,
\begin{multline*}
z_{22}w_{11} \cdot \alpha_n(Z,Z') \equiv
w_{11} \sum_{k=0}^n (-1)^{n-k} \begin{pmatrix} n \\ k \end{pmatrix}
k (z_{11})^{k-1} (z'_{11})^{n-k}  \\
= n w_{11} \sum_{l=0}^{n-1} (-1)^{n-l-1} \begin{pmatrix} n-1 \\ l \end{pmatrix}
(z_{11})^l (z'_{11})^{n-l-1}
= nw_{11} (z_{11}-z'_{11})^{n-1},
\end{multline*}
and by Theorem \ref{main2-thm},
\begin{equation}  \label{last2}
\tilde L^{(2)} \bigl( z_{22}w_{11} \cdot \alpha_n(Z,Z') \bigr) =
\tilde L^{(2)} \bigl( nw_{11} (z_{11}-z'_{11})^{n-1} \bigr)
= \frac{(-1)^{n-1}}{n^2}w_{11}(w_{11}-w'_{11})^{n-1}.
\end{equation}
Combining (\ref{last1}) and (\ref{last2}) finishes the proof.
\end{proof}

We have yet to establish that the operator $L^{(2)}$ is
$\mathfrak{gl}(2,\HC)$-equivariant.
For this reason we cannot proceed exactly as in the proof of
Proposition \ref{Lambda-prop}.
Let $\mathfrak{V} \subset \Zh \otimes {\cal H}^+$ denote the
subrepresentation of $(\varpi_2, \Zh) \otimes (\pi^0_r, {\cal H}^+)$
generated by
$$
\Bigl\{ N(Z)^{-1} \cdot \degt_Z \bigl( (z_{ij}-z'_{ij})^n \bigr);\:
n=0,1,2,3,\dots \Bigr\}.
$$
Thus we have a surjective $\mathfrak{gl}(2,\HC)$-equivariant map
$$
\mathring{L}^{(2)} : (\varpi_2 \otimes \pi^0_r, \mathfrak{V})
\twoheadrightarrow (\pi^0_l, {\cal H}^+) \otimes (\pi^0_r, {\cal H}^+).
$$

\begin{lem}
The operator $\mathring{L}^{(2)}$ annihilates
$\mathfrak{V} \cap (\Zh^+ \otimes {\cal H}^+)$.
\end{lem}

\begin{proof}
Observe that the operator $\mathring{L}^{(2)}$ increases the total
degree of an element of $\Zh \otimes {\cal H}^+$ by 2
(essentially because it involves multiplication by $N(Z-W)$).
Now, suppose that there exists an element
$x \in \mathfrak{V} \cap (\Zh^+ \otimes {\cal H}^+)$ such that
$\mathring{L}^{(2)}(x) \ne 0$. Since $\mathring{L}^{(2)}$ is
$\mathfrak{gl}(2,\HC)$-equivariant, without loss of generality we can
assume that $\mathring{L}^{(2)}(x)$ belongs to one of the irreducible components
of $(\pi^0_l, {\cal H}^+) \otimes (\pi^0_r, {\cal H}^+)$.
Furthermore, we may assume that
$$
\mathring{L}^{(2)}(x) = (z_{ij}-z'_{ij})^n \qquad \text{for some }
x \in \mathfrak{V} \cap (\Zh^+ \otimes {\cal H}^+),\: n=0,1,2,\dots.
$$
Since $(z_{ij}-z'_{ij})^n$ is homogeneous of degree $n$, only the homogeneous
component $x'$ of degree $n-2$ of $x$ contributes anything to
$\mathring{L}^{(2)}(x)$, and $x' \in \Zh^+ \otimes {\cal H}^+$.

Now, let us regard $\mathring{L}^{(2)}$ as a $U(2) \times U(2)$ equivariant
map $(\varpi_2, \Zh^+) \otimes (\pi^0_l, {\cal H}^+) \to
(\pi^0_l, \Zh^+) \otimes (\pi^0_r, {\cal H}^+)$.
We have:
$$
\mathring{L}^{(2)}(x') = (z_{ij}-z'_{ij})^n \quad
\in V_{\frac{n}2} \boxtimes V_{\frac{n}2}.
$$
Since the degree of $x'$ is $n-2$,
$$
x' \in \bigoplus_{\genfrac{}{}{0pt}{}{2l+2k+2l'=n-2}{k,l,l' \ge 0}}
N(Z)^k (V_l \boxtimes V_l) \otimes (V_{l'} \boxtimes V_{l'}).
$$
But $V_{l'} \otimes V_l$ does not contain $V_{\frac{n}2}$ unless $l+l' \ge n/2$,
which produces a contradiction.
\end{proof}

On the other hand, since $\tilde L^{(2)}$ annihilates
$(\Zh_1^- \oplus \Zh^0) \otimes {\cal H}^+$, by (\ref{LL1}) and (\ref{Z^+f}),
$\mathring{L}^{(2)}$ also annihilates $I^-_2 \otimes {\cal H}^+$.
Therefore, $\mathring{L}^{(2)}$ descends to a well-defined
$\mathfrak{gl}(2,\HC)$-equivariant quotient map
\begin{equation}  \label{surjective}
\frac{\mathfrak{V}}{\mathfrak{V} \cap ((I^-_2 \oplus \Zh^+) \otimes {\cal H}^+)}
\twoheadrightarrow {\cal H}^+ \otimes {\cal H}^+.
\end{equation}
Clearly, this quotient space is a $\mathfrak{gl}(2,\HC)$-invariant subspace of 
$(\Zh/(I^-_2 \oplus \Zh^+)) \otimes {\cal H}^+$.
Combining the fact that the map (\ref{surjective}) is surjective and
Proposition \ref{quotient-prop}, we obtain the following isomorphisms of
representations of $\mathfrak{gl}(2,\HC)$:
$$
\biggl( \varpi_2 \otimes \pi^0_r, \frac{\mathfrak{V}}{\mathfrak{V} \cap
((I^-_2 \oplus \Zh^+) \otimes {\cal H}^+)} \biggr) \simeq
\biggl( \varpi_2, \frac{\Zh}{I^-_2 \oplus \Zh^+} \biggr)
\otimes (\pi^0_r, {\cal H}^+)
\simeq (\pi^0_l, {\cal H}^+) \otimes (\pi^0_r, {\cal H}^+).
$$
We conclude that the operator $L^{(2)}$ has image in
${\cal H}^+ \otimes {\cal H}^+$
and the map (\ref{L2-2}) is indeed $\mathfrak{gl}(2,\HC)$-equivariant.
Finally, to prove the assertion about the action of $L^{(2)}$ on the
irreducible components of $(\pi^0_l, {\cal H}^+) \otimes (\pi^0_r, {\cal H}^+)$,
it is sufficient to show $L^{(2)}(x_n) = \mu_n x_n$ for some suitably chosen
generators $x_n$ of each $(\rho_n,\Zh^+ \otimes \BB C^{n \times n})$,
and this was established in Lemma \ref{main-lem}.
\end{proof}

We have the following symmetry property for the operator $L^{(2)}$,
which is a direct analogue of equation (8) in \cite{DHSS}.

\begin{lem}  \label{symmetry-lem}
The operator $L^{(2)}: (\pi^0_l, {\cal H}^+) \otimes (\pi^0_r, {\cal H}^+)
\to (\pi^0_l, {\cal H}^+) \otimes (\pi^0_r, {\cal H}^+)$
has the following symmetry:
$$
L^{(2)} (\phi_1 \otimes \phi_2)(W_1,W_2) = L^{(2)} (\phi_2 \otimes \phi_1)(W_2,W_1),
\qquad \phi_1,\phi_2 \in {\cal H}^+.
$$
\end{lem}

\begin{proof}
Clearly, this property is true for the generators $(z_{11}-z'_{11})^n$,
$n \ge 0$, of ${\cal H}^+ \otimes {\cal H}^+$.
Therefore, by the $\mathfrak{gl}(2,\HC)$-equivariance of $L^{(2)}$,
it is true for all elements of ${\cal H}^+ \otimes {\cal H}^+$.
\end{proof}

\separate

\separate

\noindent
{\em Department of Mathematics, Indiana University,
Rawles Hall, 831 East 3rd St, Bloomington, IN 47405}

\end{document}